\documentclass[12pt, a4paper, english]{amsart}

\usepackage[utf8]{inputenc}
\usepackage[T1]{fontenc}
\usepackage[english]{babel}
\usepackage{csquotes}
\usepackage{enumitem}

\usepackage[pdftex]{graphicx} 
\usepackage{booktabs} 
\usepackage{latexsym}
\usepackage{amsmath,amssymb,amsthm,amsfonts,amscd}
\usepackage{mathrsfs}
\usepackage[mathscr]{euscript}
\usepackage{color}
\usepackage{dsfont}

\usepackage[backend=bibtex, style=numeric, url=false, doi=false, hyperref=false,backref=false]{biblatex}
\usepackage{stmaryrd}
\SetSymbolFont{stmry}{bold}{U}{stmry}{m}{n}

\usepackage{mleftright}

\usepackage{tikz}			
\usetikzlibrary{arrows, angles, quotes, calc,through,backgrounds,matrix,decorations.pathmorphing}
\usepackage{multirow}

\usepackage[font=small, format=plain, labelfont=bf, up, justification=justified, singlelinecheck=false]{caption}

\usepackage{float}

\usepackage{faktor}
\usepackage{xfrac}

\usepackage{colonequals}

\usepackage{xcolor}
\usepackage{hyperref}
\usepackage[nameinlink,noabbrev, capitalise]{cleveref} 

\usepackage[title]{appendix}

\patchcmd{\subsection}{-.5em}{.5em}{}{}

\newtheoremstyle{style1}
{1cm}                   
{1cm}                   
{\normalfont}           
{}                      
{\normalfont\bfseries}  
{}                      
{ }              
{\textbf{\thmname{#1}\thmnumber{ #2} \thmnote{(#3)}}}  			

\newtheoremstyle{style2}
{1cm}                   
{1cm}                   
{\itshape }           
{}                      
{\normalfont\bfseries}  
{}                      
{ }              
{\textbf{\thmname{#1}\thmnumber{ #2} \thmnote{(#3)}}}  

\theoremstyle{style2}
\newtheorem{Theorem}{Theorem}[section]
\newtheorem{Lemma}[Theorem]{Lemma}	
\newtheorem{Corollary}[Theorem]{Corollary}	
\newtheorem{Proposition}[Theorem]{Proposition}

\theoremstyle{style1}
\newtheorem{Definition}[Theorem]{Definition} 
\newtheorem{Remark}[Theorem]{Remark}	
\newtheorem{Example}[Theorem]{Example}
\newtheorem{Convention}[Theorem]{Convention}	

\numberwithin{equation}{section} 

\setlist[enumerate]{topsep=0.65ex, itemsep=0.5ex}

\newcommand{\HH}{\mathbb{H}} 
  
\newcommand{\NN}{\mathbb{N}} 
\newcommand{\RR}{\mathbb{R}}
\newcommand{\ZZ}{\mathbb{Z}}

\DeclareMathOperator{\ad}{\mathrm{ad}}
\DeclareMathOperator{\homdim}{\mathrm{homdim}}

\DeclareMathOperator{\cA}{\mathcal{A}}

\DeclareMathOperator{\cF}{\mathcal{F}}
\DeclareMathOperator{\cH}{\mathcal{H}}
\DeclareMathOperator{\cP}{\mathcal{P}}
\DeclareMathOperator{\cS}{\mathcal{S}}
\DeclareMathOperator{\fg}{\mathfrak{g}}

\renewcommand{\epsilon}{\varepsilon}

\title[Complexity of Cut-and-Project Sets]{\textbf{Complexity of Cut-and-Project Sets of Polytopal Type in Special Homogeneous Lie Groups}}

\author{Peter Kaiser}
\address{Institut für Algebra und Geometrie, KIT, Karlsruhe, Germany}
\email{peter.kaiser2@kit.edu}
\urladdr{https://www.math.kit.edu/didaktik/~kaiserp/de}

\begin{document}
	\setlength{\parindent}{0mm}
	\renewcommand{\thesection}{\arabic{section}}
	\renewcommand{\thesubsection}{\thesection.\arabic{subsection}}
	
	\begin{abstract}
		The aim of this paper is to determine the asymptotic growth rate of the complexity function of cut-and-project sets in the non-abelian case. In the case of model sets of polytopal type in homogeneous two-step nilpotent Lie groups we can establish that the complexity function asymptotically behaves like $r^{\homdim(G) \dim(H)}$. Further we generalize the concept of acceptance domains to locally compact second countable groups.
	\end{abstract}
	\maketitle
	
	\section{Introduction}
	This article is concerned with the complexity of discrete subsets of locally compact groups, which obey some form of aperiodic order.\par 
	For discrete subsets of locally compact abelian groups, notably for discrete subsets of $\RR^n$, there is a established notion of complexity based on the study of the so-called patch counting function \cite{ArnouxMauditShiokawa, Baryshnikov, Julien, KoivusaloWalton1, Lagarias2, Lagarias3, Moody1, Moody2, Vuillon}. More recently, there has been an approach to extend results about discrete subsets of locally abelian groups to general locally compact groups \cite{Beckus, BjorklundHartnick1, BjorklundHartnick2, BjorklundHartnick3, LenzStrungaru, Schlottmann1, Schlottmann2}.\par
	In the present article we contribute to this program by extending the notion of complexity to discrete subsets of non-abelian locally compact group. More specifically, we are going to generalize an approach of Julien, \cite{Julien}, and Koivusalo and Walton, \cite{KoivusaloWalton1}. While the theory works in full generality, we will obtain our strongest results in the case of two-step-nilpotent Lie groups. 
	
	\subsection{Aperiodic order in the Euclidean case}
	Consider the abelian group $(\RR^n, +)$ as a metric group with respect to the standard Euclidean metric. A set $\Lambda \subset \RR^n$ is called \textit{locally finite} if for all bounded sets $B \subset \RR^n$ the intersection $\Lambda \cap B$ is finite. For this sets one can define the patch counting function $p(r)$ (see \cref{Def:p}) as  a measure of their complexity. Examples of locally finite sets are lattices. Their complexity functions are constant 1, meaning that lattices are highly structured. In the case of aperiodic ordered sets the patch counting function is growing at least linearly \cite{Lagarias2, Lagarias3, Moody1, Moody2, Vuillon}. A locally finite set with $p(r)< \infty$ for all $r>0$ is called a set with \textit{finite local complexity} or an \textit{FLC set}.\par\medskip
	
	There are two important methods to construct FLC sets, either by \textit{substitution} or by \textit{cut-and-project}. We are interested in the cut-and-project approach, which is due to Yves Meyer, who is a pioneer in the field of aperiodic order and has laid the foundation for a lot of our common knowledge in the 1960s \cite{Meyer1, Meyer2, Meyer3}. The idea in the cut-and-project approach is to consider a lattice $\Gamma$ in the product $\RR^n \times \RR^d$. Then one chooses a subset $W \subset \RR^d$, which is called the \textit{window}. The projection of $(\RR^n \times W) \cap \Gamma$ to $\RR^n$ results in a point set, which is called a \textit{cut-and-project set}. Under some extra conditions this cut-and-project set is an FLC set, in this case it is called the \textit{model set} defined by the data $\Lambda(\RR^n, \RR^d, \Gamma, W)$. Such sets have been studied from different perspectives, \cite{BaakeHuckStrungaru, KoivusaloWalton1, Hof3, Lagarias1, Moody3}. \par\medskip
	
	In the 1980s the popularity of this field was pushed by the discovery of \textit{quasi-crystals}, \cite{Shechtman}. After this discovery physicists, crystallographers and mathematicians worked on models to describe these newly discovered aperiodic structures. Physicists are primarily interested in quasi-crystals in $\RR^n$, for $n \leq 3$, but mathematically the restriction on the dimension is unnatural and therefore was rapidly dropped. A history of the developments in this time can be found in the book by Senechal, \cite{Senechal}. The characterising property of a quasi-crystal is \textit{pure-point diffraction}, which is a global property and was studied in \cite{BaakeLenz, Dworkin, Hof1, Hof2, Lagarias2b}. Another line of research is to characterise the structure of an aperiodic ordered set by some local data, namely its \textit{repetitivity} or its \textit{complexity}, \cite{Lagarias2, Lagarias3, Moody1, Moody2, Vuillon}. For a comprehensive overview of the field see \cite{BaakeGrimm}.\par \medskip 
	
	This paper will focus on understanding the complexity of model sets. We want to determine how the \textit{complexity function $p(r)$} behaves asymptotically. 
	\begin{Definition}[Patch]
		Let $X$ be a metric space, $\Lambda \subset X$ a locally finite subset, $\lambda \in \Lambda$ and $r\in \RR^+$. Then the \textit{$r$-patch} $P_r(\lambda)$ is the constellation of points from $\Lambda$ around $\lambda$, which have distance at most $r$ to $\lambda$, i.e. $ P_r(\lambda):= B_r(\lambda) \cap \Lambda$.\par\medskip
		
		If $G$ is a lcsc group the set of patches of radius $r$ impose an equivalence relation on the elements of $\Lambda \subset G$ by
		\begin{equation} 
			\label{Def:requiv}
			\lambda \sim_r \mu :\Leftrightarrow P_r(\lambda)\lambda^{-1}=P_r(\mu)\mu^{-1}.
		\end{equation}
		We will denote the \textit{$r$-equivalence class of $\lambda$} by
		\begin{equation*}
			A_r^G(\lambda):= \{\mu \in \Lambda \mid \lambda \sim_r \mu\} \subset G
		\end{equation*}
		and the set of all equivalence classes by 
		\begin{equation*}
			A_r^G:=\{A_r^G(\lambda) \mid \lambda \in \Lambda\}.
		\end{equation*}
	\end{Definition}
	
	\begin{Definition}[Complexity function]\label{Def:p}
		Let $G$ be a lcsc group and $\Lambda \subset G$ a locally finite subset. Then the \textit{complexity function} $p(r)$ is given by
		\begin{equation*}
			p(r) := \left\vert\left\{B_r(e) \cap \Lambda \lambda^{-1} \,\middle\vert\, \lambda \in \Lambda \right\} \right\vert = \left\vert\left\{P_r(\lambda) \lambda^{-1} \,\middle\vert\, \lambda \in \Lambda \right\}\right\vert = \left\vert A_r^G \right\vert . 
		\end{equation*}
	\end{Definition}

	The function $p$ is also called the \textit{patch-counting function} and first appears in the work by Lagarias and Pleasants, \cite{Lagarias3}. Note that model sets carry information than the underlying point set itself, this can be used to determine the complexity function. Some early work in this context is done in \cite{ArnouxMauditShiokawa} and \cite{Baryshnikov} for some special cases and low dimensions. A general approach first appeared in the paper by Julien, \cite{Julien}, the main idea is that each class of patches corresponds to a certain region inside the window, the so called \textit{acceptance domain}. Optimal results can be obtained in the case of \textit{polytopal windows}, that is when $W$ is a convex polytope. The ideas of Julien where picked up by Koivusalo and Walton, \cite{KoivusaloWalton2}, who where proved the following theorem. We will assume the stabilizers of the hyperplanes which bound the window are trivial; in the original theorem the role of this stabilizers is addressed.
	
	\begin{Theorem}[Koivusalo, Walton, {\cite[Theorem 7.1]{KoivusaloWalton2}}]
		Consider a model set $\Lambda(\RR^n, \RR^d, \Gamma, W)$ with a polytopal window $W$. Assume that the stabilizer of the hyperplanes which bound the window are trivial. Then the complexity grows asymptotically as $p(r) \asymp r^{n \cdot d}$.
	\end{Theorem}
	
	\subsection{Aperiodic order beyond the Euclidean case}
	A natural generalisation of FLC sets in $\RR^n$ is to consider FLC sets in arbitrary \textit{locally compact groups} equipped with some metric. We will be interested in studying their complexity functions. We emphasise that in doing so the choice of metric is important. By the restriction to \textit{metric locally compact second countable (lcsc)} groups a theorem of Struble, \cite{Struble}, guarantees the existence of a `nice' metric. `Nice' means in this context that the metric is right-invariant, proper and compatible. This is the setup in which the euclidean ideas are generalized \cite{Beckus, BjorklundHartnick1, BjorklundHartnick2, BjorklundHartnick3, LenzStrungaru}.\par \medskip 
	
	The cut-and-project approach also applies in this more general setup, \cite{BjorklundHartnick3,Schlottmann1, Schlottmann2}. The question we want to answer is how the approach of Julien and Koivusalo and Walton can be translated to this more general set-up?
	
	\subsection{Results on two-step homogeneous Lie groups}
	Ideally one would like to describe the complexity of FLC sets for all lcsc groups. However this turns out to be quite challenging so we will have to introduce some more restrictions. In particular, since we want to follow the approach of Julien, \cite{Julien}, and Koivusalo and Walton, \cite{KoivusaloWalton2}, we need a notion of hyperplanes. So a first question is, in which groups can we define hyperplanes?\par 
	We will consider homogeneous Lie groups. These groups are nilpotent, real, finite-dimensional, connected, simply connected and admit a family of dilations which replace the scalar multiplication. For a detailed discussion of such groups we refer to the book by Fischer and Ruzhansky, \cite{FischerRuzhansky}. For this class of groups it is possible to identify the underlying set of the Lie group $G$ with the corresponding Lie algebra $\fg$. Since $\fg$ is a vector space we can define hyperplanes in the usual sense.\par 
	Moreover these groups admit a canonical quasi-isometry class of homogeneous norms, which provide the same complexity resolving the aforementioned issue of dependence of the choice of metric. It turns out that balls with respect to such norms have exact polynomial growth, i.e. the volume of a ball $B_r(e)$ grows as $r^{\alpha}$. The exponent of this growth is called the \textit{homogeneous dimension} of the homogeneous Lie group.\par\medskip 
	
	A second restriction has to be made since we also need that the group acts on the space of hyperplanes in the vector space underlying $\fg$. We can show that this is the case exactly if the Lie group has nilpotency degree one or two, i.e. if it is abelian or two-step nilpotent. For higher nilpotency degree the action of the group bends the hyperplanes into algebraic hypersurfaces.\par\medskip
	Naively one would expect that the complexity function of a model set $\Lambda(G,H, \Gamma, W)$ would depend on the dimension of the Lie groups $G$ and $H$, i.e. $p(r)\asymp r^{\dim(G)\dim(H)}$, or if not their homogeneous dimensions , i.e. $p(r) \asymp r^{\homdim(G)\homdim(H)}$, but surprisingly both turn out to be false. In fact the two factors behave differently, on the $G$-side the homogenous dimension replaces the dimension, while on the $H$-side it does not. More precisely, we prove the following theorem, which is the main theorem of this paper.
	
	\begin{Theorem}[Informal version of the main theorem]
		Consider a model set $\Lambda(G, H, \Gamma, W)$ with a convex polytopal window $W$ and $G$ and $H$ two-step nilpotent homogeneous Lie groups. Assume that the stabilizer of the hyperplanes which bound the window are trivial. Then the complexity grows asymptotically as $p(r) \asymp r^{\homdim(G) \cdot \dim(H)}$.
	\end{Theorem}
	
	\subsection{Method of proof}
	The proof of the main theorem consists of four steps. The first three are similar to the Euclidean case, while the forth one is uses different techniques.\par\medskip
	
	First we will establish the connection between the equivalence classes of patches and the acceptance domains in \cref{Sec:Complexity}. This is a translation from the Euclidean case considered in \cite{KoivusaloWalton2}. The only difference is that we have to be a bit more careful since our groups are in general non-abelian. The established result is the same as in the Euclidean case.
	\begin{Theorem}[Acceptance domains vs. Equivalence classes]
		Let $\Lambda(G,H, \Gamma, W)$ be a model set, with $G, H$ lcsc groups, $\Gamma \subset G \times H$ a uniform lattice and $W\subset H$ a non-empty, pre-compact, $\Gamma$-regular window, then
		\begin{equation*}
			A_r^H(\lambda)  \subset \left(\bigcap_{\mu \in \cS_r(\lambda)}\mu \mathring{W}\right)\cap \left(\bigcap_{\mu \in \cS_r(\lambda)^\mathrm{C}} \mu W^\mathrm{C} \right)=:W_r(\lambda).
		\end{equation*}
		The $W_r(\lambda)$ are called \textit{$r$-acceptance domains of $\lambda$}. Further for $\lambda \not\sim_r \lambda'$ we have
		\begin{equation*}
			W_r(\lambda) \cap W_r(\lambda') = \emptyset. 
		\end{equation*}
		Finally we have
		\begin{equation*}
			\overline{W} = \bigcup_{\lambda \in A_r^G}\overline{W_r(\lambda)}.
		\end{equation*}
	\end{Theorem}
	We will give the precise definition of $\cS_r(\lambda)$ in \cref{Sec:Complexity} and of $A_r^H(\lambda)$ in \cref{Def:Pre}. For now think of $A_r^H(\lambda)$ as the projection of all the  points in the equivalence class of $\lambda$ to $H$. Further $\cS_r(\lambda)$ and $\cS_r^\mathrm{C}(\lambda)$ are roughly speaking a decomposition of the possible neighbours of $\lambda$ projected to $H$.\par\medskip
	
	As a second step we establish a lattice point counting argument in \cref{Sec:Lattice}.
	\begin{Proposition}[Growth Lemma]
		Let $G$ and $H$ be lcsc groups. For a model set $\Lambda(G,H, \Gamma, W)$ with a uniform lattice $\Gamma \subset G \times H$ and a bounded open set ${\emptyset \neq A \subset H}$. The asymptotic growth of the number of lattice points inside $B^G_r(e) \times A$ is bounded by
		\begin{equation*} 
			\mu_G\left(B^G_{r-k_1}(e)\right) \ll \left\vert(B^G_r(e)\times A)\cap \Gamma \right\vert \ll \mu_G\left(B^G_{r+k_2}(e)\right),
		\end{equation*}
		for some constants $k_1,k_2>0$ as $r \to \infty$.
	\end{Proposition} 
	
	\begin{Remark}
		For the asymptotic behaviour we use the common notation $g(t) \ll f(t)$ which means $\limsup\limits_{t \to \infty}\left\vert \frac{g(t)}{f(t)} \right\vert < \infty$. If both $g(t) \ll f(t)$ and $g(t) \gg f(t)$ holds we write $g(t) \asymp f(t)$.
	\end{Remark}
	
	This result connects the number of lattice points in sets of a certain form to the measure of these sets. The standard proof is via ergodic theory, but we will give a more elementary proof.\par\medskip
	
	In the third step we show in \cref{Sec:Growth} that we can estimate the number of acceptance domains by extending the boundary of the window. This is also done in \cite{KoivusaloWalton2}, the new regions inside the window are called \textit{cut regions} in this paper. Again the result we obtain is the same as in the Euclidean case. But we have to overcome a major difference in its proof, since in the Euclidean case the group acts on a hyperplane by translations, which preserves the directions of the hyperplane. Or formulated differently, a translation does not rotate a hyperplane. In our general setup the group action can rotate hyperplanes, leading to a new phenomenon.\par 
	\begin{Theorem}[Cut regions vs. Acceptance domains]
		For a polytopal model set $\Lambda(G,H,\Gamma,W)$, where the window is bounded by $P_1,...,P_N$, and $G$ and $H$ at most two-step nilpotent homogeneous Lie groups we have
		\begin{equation*}
			\left\vert A_r^H\right\vert \leq \# \pi_0 \left( H \setminus \bigcup_{\mu\in\cS_r} \bigcup_{i=1}^N \mu P_i\right).
		\end{equation*}
		And for a certain ball $B_h(c_W)\subset W$ we also have
		\begin{equation*}
			\# \pi_0 \left( B_h(c_W) \setminus \bigcup_{i=1}^N \bigcup_{\mu \in U_i(r)} \mu P_i \right) \leq \left\vert A_r^H \right\vert.
		\end{equation*}
	\end{Theorem}
	The last step is devoted to solving the problem with the rotating hyperplanes see \cref{Sec:Combinatorics}. To do so we use the theory of hyperplane arrangements; this tool was not needed in the abelian case. The research of such arrangements has a long history and goes back as far as \cite{Schlafli}, more modern approaches are due to Grünbaum \cite{Grunbaum1, Grunbaum2, Grunbaum3} and Zaslavsky \cite{Zaslavsky}. Two sources for a survey of the field are the book by Dimca, \cite{Dimca}, and the lecture notes by Stanley, \cite{Stanley}. An important tool for our combinatorial argument is by Beck, \cite{Beck}. We need a special version of Beck's theorem. To prove this version of Beck's theorem we need some combinatorial inputs from \cite{Szekely} and \cite{SzemerediTrotter}. This lets us extend the standard counting formulas to our specific context and we can proof the following theorem, which will then finish the proof of the main theorem.
	\begin{Theorem}[Higher dimensional local dual of Beck's Theorem]
		Let $\cH$ be a hyperplane arrangement in $\RR^d$, let $B \subset \RR^d$ be convex and let $c_d$ be a constant only depended on the dimension. Further let $\cH$ consist of $d$ families $F_1,...,F_d$ with $\vert F_i \vert = \frac{n}{d}$ and such that for all $(f_1,...,f_d) \in F_1 \times ... \times F_d$ we have $B \cap \bigcap_{i=1}^d f_i = \{p\}$ for some point $p \in B$. Moreover assume that there is $c< \frac{1}{100}$ such that at most $c \cdot \vert F_i \vert$ hyperplanes from $F_i$ can intersect in one point. Then the number of intersection points in $B$ exceeds $c_d \cdot n^d$, i.e. $\vert F_{0,B} \vert \geq c_d \cdot n^d$.
	\end{Theorem}
	
	\subsection{General results for lcsc groups}
	As discussed above our restriction to homogeneous Lie groups is necessary for the proof of the main theorem, but the individual steps work in greater generality. Acceptance domains and cut regions can be defined for all connected lcsc groups, the polytopal condition on the window is not needed for this approach.
	
	\subsection{Notation}
	Throughout this text, $G$ and $H$ will always be locally compact second countable groups, $\Gamma \subset G \times H$ a uniform lattice and $W \subset H$ a precompact $\Gamma$-regular, i.e. $\partial W \cap \pi_H(\Gamma) = \emptyset$, subset with nonempty interior, which is called the \textit{window}. Why these restrictions on $W$ are required will be explained in \cref{Prop:CPSFLCrelativelydense} below. Further we denote the projection on $G$ by $\pi_G$ and on $H$ by $\pi_H$.
	
	\begin{Definition}
		A triple $(G,H, \Gamma)$ is called a \textit{cut-and-project scheme (CPS)} if $\pi_G\vert_\Gamma$ is injective and $\pi_H(\Gamma)$ is dense in $H$. The set
		\begin{equation*}
			\Lambda = \Lambda(G, H, \Gamma, W) := \pi_G((G \times W) \cap \Gamma) = \tau^{-1}(\Gamma_H \cap W)
		\end{equation*}
		is called a \textit{model set} if $(G,H, \Gamma)$ is a cut-and-project scheme. Here we use the notation $\Gamma_H := \pi_H(\Gamma)$ and $\tau:= \pi_H \circ (\pi_G \vert_\Gamma)^{-1}$.
	\end{Definition}
	
	\begin{Remark}
		Observe that we do not consider non-uniform model sets, so in our terminology of a model set the lattice is always uniform.
	\end{Remark}
	
	We will always put a $G$, resp. $H$ in the index if we consider the projection of an object to the factor $G$, resp. $H$. The diagram in \cref{Fig:CPS} visualizes the relation between the different groups in this setup.\par\medskip
	
	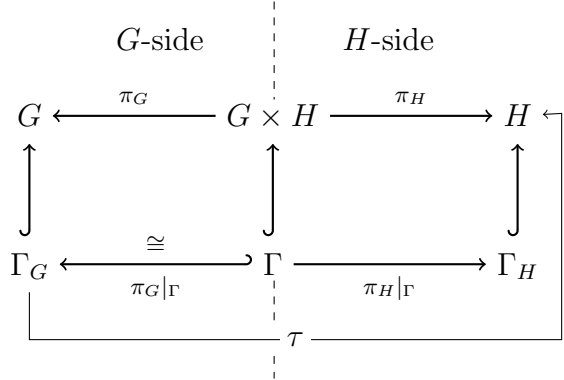
\begin{figure}[h]
		\centering
		\captionsetup{width=0.45\linewidth}
		\begin{tikzpicture}[descr/.style={fill=white,inner sep=2.5pt}]
			\matrix (m) [matrix of math nodes, row sep=3em,
			column sep=5em, text height=2ex, text depth=0.5ex]
			{ G & G \times H & H \\
				\Gamma_G	& \Gamma & \Gamma_H\\ };
			
			\path[thick, ->,font=\scriptsize] (m-1-2) edge node[above] {$\pi_G$} (m-1-1)
			edge node[above] {$\pi_H$}(m-1-3)
			(m-2-2)	edge node[below] {${\pi_H}\vert_\Gamma$} (m-2-3);
			\path[thick,right hook->,font=\scriptsize] (m-2-1) edge (m-1-1)
			(m-2-2) edge (m-1-2)
			(m-2-3) edge (m-1-3);
			\path[thick,left hook->,font=\scriptsize] (m-2-2) edge node[below] {${\pi_G}\vert_\Gamma$} node[above] {$\cong$} (m-2-1);	
			
			\draw[->] (m-2-1) -- ++ (0,-1) -- node[descr] {$\tau$} ++ (7,0) -- ++ (0,3) -- (m-1-3);
			\node at (-1.5,2) {$G$-side};
			\node at (1.5,2) {$H$-side};
			\draw[-, dashed] (0, 2.5) -- (0, 1.2);
			\draw[-, dashed] (0, -1.2) -- (0, -1.8);
			\draw[-, dashed] (0, -2.2) -- (0, -2.5);
		\end{tikzpicture}
		\caption{Visualisation of a CPS.}
		\label{Fig:CPS}
	\end{figure}
	
	To study complexity we need a metric, which will always be associated to a norm by $d(x,y):= \vert x y^{-1}\vert$. In the case of homogeneous Lie groups, we have a homogeneous norm on $G$ and $H$, which we will denote by $\vert \cdot \vert_G$, $\vert \cdot \vert _H$, if it is clear which norm we consider we drop the index. 
	
	\begin{Definition}\textbf{(Delone)}
		Let $X$ be a metric space, a subset $\Lambda \subset X$ is called \textit{$(R,r)$-Delone} if:
		\begin{enumerate}
			\item It is \textit{$R$-relatively dense}, i.e.
			\begin{equation*}
				\exists R>0\, \forall x \in X : B_R(x) \cap \Lambda \neq \emptyset. 
			\end{equation*}
			\item It is \textit{$r$-uniformly discrete}, i.e.
			\begin{equation*}
				\exists r>0\, \forall \lambda, \mu \in \Lambda: d(\lambda,\mu) \geq r.
			\end{equation*}
		\end{enumerate}
		If one is not interested in the parameters $R$ and $r$ one simply speaks of a \textit{Delone set}.
	\end{Definition}
	
	\begin{Remark}
		By \cref{Prop:CPSFLCrelativelydense} model sets are Delone sets.
	\end{Remark}
	
	\begin{Definition}[Pre-Acceptance domains]\label{Def:Pre}
		Let $\Lambda(G,H, \Gamma, W)$ be a model set.	The image of $A_r^G(\lambda)$ under the map $\tau$ is called the \textit{$r$-pre-acceptance domain of $\lambda$}
		\begin{equation*}
			A_r^H(\lambda):=\tau\left(A_r^G(\lambda)\right) \subset H.
		\end{equation*}
		We denote the set of all pre-acceptance domains by $A_r^H$.
	\end{Definition}
	
	\section{How to measure complexity?}\label{Sec:Complexity}
	In this section we will explain, how one can determine the complexity function of a given FLC set, e.g. a model set $\Lambda(G,H, \Gamma,W)$, where $G=(G,d_G)$ and $H=(H,d_H)$ are metric locally compact second countable groups. The statements in this section are translations from the euclidean case, we refer to the paper of Koivusalo and Walton, \cite{KoivusaloWalton2}, for this set-up. We will take a closer look on how the complexity function $p$ is connected to local constellations, which are called \textit{patches}, in the FLC set. We will then establish a connection between these patches and certain regions in the window $W$.\par\medskip

	\begin{Definition}[Finite local complexity]\label{Rem:FLC}
		Let $\Lambda \subset G$ be a FLC subset. If the complexity function of $\Lambda$ is finite for all $r$ we say that $\Lambda$ has \textit{finite local complexity}. There are several ways of viewing this condition, compare \cref{FLC_Lemma}.
	\end{Definition}
	
	\begin{Theorem}[Acceptance domains]\label{Thm:AcceptanceDomains}
		Let $\Lambda(G,H, \Gamma, W)$ be a model set, with $G, H$ lcsc groups, $\Gamma \subset G \times H$ a uniform lattice and $W\subset H$ a non-empty, pre-compact, $\Gamma$-regular window, then
		\begin{equation}\label{Thm:AcceptanceDomainsEq1}
			A_r^H(\lambda)  \subset \left(\bigcap_{\mu \in \cS_r(\lambda)}\mu \mathring{W}\right)\cap \left(\bigcap_{\mu \in \cS_r(\lambda)^\mathrm{C}} \mu W^\mathrm{C} \right)=:W_r(\lambda),
		\end{equation}
		where $\cS_r(\lambda)$ will be defined in \cref{Def:Slab}.
		The $W_r(\lambda)$ are called \textit{$r$-acceptance domains of $\lambda$}. Further for $\lambda \not\sim_r \lambda'$ we have
		\begin{equation}\label{Thm:AcceptanceDomainsEq2}
			W_r(\lambda) \cap W_r(\lambda') = \emptyset. 
		\end{equation}
		Finally we have
		\begin{equation}\label{Thm:AcceptanceDomainsEq3}
			\overline{W} = \bigcup_{\lambda \in A_r^G}\overline{W_r(\lambda)}.
		\end{equation}
	\end{Theorem}
	
	\begin{Remark}
		The terminology is due to Koivusalo and Walton, \cite{KoivusaloWalton2}. In their paper they treat the case of model sets $\Lambda(\RR^d,\RR^n,\Gamma, W)$ and extend Julien's paper, \cite{Julien}, which first introduced the idea of considering a decomposition of the window.
	\end{Remark}
	
	\begin{Corollary}
		$p(r)=\vert \{W_r(\lambda) \mid \lambda \in \Lambda\}\vert=\left\vert A_r^H\right\vert$.
	\end{Corollary}
	
	The rest of the section is devoted to the proof of the theorem and we begin by working towards the definition of $S_r(\lambda)$.
	
	\begin{Definition}[Displacements]
		Let $\Lambda$ be a CPS. We define the displacements of $\lambda \in \Lambda$ as
		\begin{equation*}
			\mathrm{Disp}(\lambda):=\{\mu \in \Gamma_G \mid \mu \lambda \in \Lambda\}.
		\end{equation*}
	\end{Definition}
	
	\begin{Lemma}[{\cite[Lemma 2.1]{KoivusaloWalton2}}]\label{Lem:WindowShift}
		Let $\lambda \in \Lambda$ and $\mu \in G$, if $\mu \lambda \in \Lambda$ then $\mu \in \Gamma_G$. On the other hand if $\mu \in \Gamma_G$: 
		\begin{equation*}
			\mu \lambda \in \Lambda \;\Leftrightarrow\; \tau(\lambda) \in \tau(\mu)^{-1}\mathring{W} \;\Leftrightarrow\; \tau(\mu) \in \mathring{W}\tau(\lambda)^{-1}.
		\end{equation*}
		In particular $\tau(\mathrm{Disp}(\lambda)) \subset \mathring{W}\mathring{W}^{-1}$.
		\begin{proof}
			Since $\lambda, \mu \lambda \in \Lambda$ we find elements $\gamma, \delta \in \Gamma$ such that $\pi_G(\gamma)=\lambda$, $\pi_G(\delta)=(\mu \lambda)^{-1}$. Then $\gamma \delta \in \Gamma$ and $\pi_G(\gamma \delta)= \lambda (\mu \lambda)^{-1} = \mu^{-1} \in \Gamma_G$ and therefore $\mu \in \Gamma_G$.\par 
			Now let $\mu \in \Gamma_G$. By definition $\mu \lambda \in \Lambda$ if and only if $\tau(\mu \lambda) \in \mathring{W}$ and since $\tau$ is a homomorphism this is equivalent to $\tau(\mu) \in \mathring{W} \tau(\lambda)^{-1}$ and $\tau(\lambda) \in \tau(\mu)^{-1}\mathring{W}$.
		\end{proof}
	\end{Lemma}
	
	To understand patches on the $H$-side of the model set we transport the information of the displacements to this side. Since we always consider patches in dependence of $r$ we only need displacements of magnitude at most $r$.
	
	\begin{Definition}[$r$-slab]\label{Def:Slab}
		Let $\Lambda(G,H,\Gamma,W)$ be a model set. We define the \textit{$r$-slab} as
		\begin{equation*}
			\cS_r := \pi_H\left(\left\{(\gamma,\mu) \in \Gamma \,\middle\vert\,  \vert \gamma \vert <r \text{ and } \mu \in WW^{-1} \right\}\right).
		\end{equation*}
		Further in the case that we only are interested in the displacements of a certain equivalence class we define the \textit{$r$-slab of $\lambda$} as
		\small\begin{equation*}
			\cS_r(\lambda) := \pi_H\left(\left\{(\gamma,\mu) \in \Gamma \,\middle\vert\, \vert \gamma \vert <r \text{ and } \mu \in WW^{-1} \text{ and } \gamma^{-1} \in \mathrm{Disp}(\lambda) \right\}\right)
		\end{equation*}\normalsize
		and 
		\small\begin{equation*}
			\cS_r^\mathrm{C}(\lambda) := \pi_H\left(\left\{(\gamma,\mu) \in \Gamma \,\middle\vert\, \vert \gamma \vert <r \text{ and } \mu \in WW^{-1} \text{ and } \gamma^{-1} \notin \mathrm{Disp}(\lambda) \right\}\right).
		\end{equation*}\normalsize
	\end{Definition}
	
	\begin{Remark}
		In the paper of Koivusalo and Walton the sets $S_r(\lambda)$ and $S_r^\mathrm{C}(\lambda)$ are called $P_{in}$ and $P_{out}$, we think that our notation highlights the connection to the slab in a better way. Whereas their notation highlights the connection to the patch $P$.
	\end{Remark}
	
	\begin{figure}
		\centering
		\captionsetup{width=0.75\linewidth}
		\begin{tikzpicture}
			\draw[->] (-1.5,0) -- (10,0);
			\node at (10,-0.3) {$G$};
			\draw[->] (0,-1.5) -- (0,5);
			\node at (-0.3,5) {$H$};
			\filldraw [gray] 
			(-1.8284+1/2, 3.8284+1/2) circle (2pt)
			
			(-0.4142-1.4142/2, 2.4142+1.4142/2) circle (2pt)
			(-0.4142-0.4142/2, 2.4142+2.4142/2) circle (2pt)
			(-0.4142-0.4142/2+1/2, 2.4142+3.4142/2) circle (2pt)
			
			(-1.4142, 1.4142) circle (2pt)
			(-1.4142+1/2, 1.4142+1/2) circle (2pt)
			(-0.4142, 2.4142) circle (2pt)
			(-0.4142+1/2, 2.4142+1/2) circle (2pt)
			( 0.5858, 3.4142) circle (2pt)
			( 0.5858+1/2, 3.4142+1/2) circle (2pt)
			( 1.5858, 4.4142) circle (2pt)
			
			(-2.4142/2, 0.4142/2) circle (2pt)
			(-1.4142/2, 1.4142/2) circle (2pt)
			(-0.4142/2, 2.4142/2) circle (2pt)
			( 0.5858/2, 3.4142/2) circle (2pt)
			( 1.5858/2, 4.4142/2) circle (2pt)
			( 2.5858/2, 5.4142/2) circle (2pt)
			( 3.5858/2, 6.4142/2) circle (2pt)
			( 4.5858/2, 7.4142/2) circle (2pt)
			( 5.5858/2, 8.4142/2) circle (2pt)
			
			(-1,-1) circle (2pt)
			(-1/2,-1/2) circle (2pt)
			(0,0) circle (2pt)
			(1/2,1/2) circle (2pt)
			(1,1) circle (2pt)
			(3/2,3/2) circle (2pt)
			(2,2) circle (2pt)
			(5/2,5/2) circle (2pt)
			(3,3) circle (2pt)
			(7/2,7/2) circle (2pt)
			(4,4) circle (2pt)	
			
			(1.4142,-1.4142) circle (2pt)
			(1.4142+1/2,-1.4142+1/2) circle (2pt)
			(2.4142,-0.4142) circle (2pt)
			(2.4142+1/2,-0.4142+1/2) circle (2pt)
			(3.4142, 0.5858) circle (2pt)
			(3.4142+1/2, 0.5858+1/2) circle (2pt)
			(4.4142, 1.5858) circle (2pt)
			(4.4142+1/2, 1.5858+1/2) circle (2pt)
			(5.4142, 2.5858) circle (2pt)
			(5.4142+1/2, 2.5858+1/2) circle (2pt)
			(6.4142, 3.5858) circle (2pt)
			(6.4142+1/2, 3.5858+1/2) circle (2pt)
			
			(0.4142/2,-2.4142/2) circle (2pt)
			(1.4142/2,-1.4142/2) circle (2pt)
			(2.4142/2,-0.4142/2) circle (2pt)
			(3.4142/2, 0.5858/2) circle (2pt)
			(4.4142/2, 1.5858/2) circle (2pt)
			(5.4142/2, 2.5858/2) circle (2pt)
			(6.4142/2, 3.5858/2) circle (2pt)
			(7.4142/2, 4.5858/2) circle (2pt)
			(8.4142/2, 5.5858/2) circle (2pt)
			(9.4142/2, 6.5858/2) circle (2pt)
			(10.4142/2, 7.5858/2) circle (2pt)
			(11.4142/2, 8.5858/2) circle (2pt)
			
			(4.8284-1.4142/2-1,-0.8284+1.4142/2-1) circle (2pt)
			(4.8284-1.4142/2-1/2,-0.8284+1.4142/2-1/2) circle (2pt)
			(4.8284-1.4142/2,-0.8284+1.4142/2) circle (2pt)
			(4.8284-1.4142/2+1/2,-0.8284+1.4142/2+1/2) circle (2pt)
			(4.8284-1.4142/2+2/2,-0.8284+1.4142/2+2/2) circle (2pt)
			(4.8284-1.4142/2+3/2,-0.8284+1.4142/2+3/2) circle (2pt)
			(4.8284-1.4142/2+4/2,-0.8284+1.4142/2+4/2) circle (2pt)
			(4.8284-1.4142/2+5/2,-0.8284+1.4142/2+5/2) circle (2pt)
			(4.8284-1.4142/2+6/2,-0.8284+1.4142/2+6/2) circle (2pt)
			(4.8284-1.4142/2+7/2,-0.8284+1.4142/2+7/2) circle (2pt)
			(4.8284-1.4142/2+8/2,-0.8284+1.4142/2+8/2) circle (2pt)
			
			(4.8284-1/2,-0.8284-1/2) circle (2pt)
			(4.8284,-0.8284) circle (2pt)
			(4.8284+1/2,-0.8284+1/2) circle (2pt)
			(5.8284, 0.1716) circle (2pt)
			(5.8284+1/2, 0.1716+1/2) circle (2pt)
			(6.8284, 1.1716) circle (2pt)
			(6.8284+1/2, 1.1716+1/2) circle (2pt)
			(7.8284, 2.1716) circle (2pt)
			(7.8284+1/2, 2.1716+1/2) circle (2pt)
			(8.8284, 3.1716) circle (2pt)
			(8.8284+1/2, 3.1716+1/2) circle (2pt)
			(9.8284, 4.1716) circle (2pt)
			
			(7.2426-1.4142/2-1/2, -1.2426+1.4142/2-1/2) circle (2pt)
			(7.2426-1.4142/2, -1.2426+1.4142/2) circle (2pt)
			(7.2426-1.4142/2+1/2, -1.2426+1.4142/2+1/2) circle (2pt)
			(7.2426-1.4142/2+2/2, -1.2426+1.4142/2+2/2) circle (2pt)
			(7.2426-1.4142/2+3/2, -1.2426+1.4142/2+3/2) circle (2pt)
			(7.2426-1.4142/2+4/2, -1.2426+1.4142/2+4/2) circle (2pt)
			(7.2426-1.4142/2+5/2, -1.2426+1.4142/2+5/2) circle (2pt)
			
			(7.2426, -1.2426) circle (2pt)
			(7.2426+1/2, -1.2426+1/2) circle (2pt)
			(8.2426, -0.2426) circle (2pt)
			(8.2426+1/2, -0.2426+1/2) circle (2pt)
			(9.2426,  0.7574) circle (2pt)
			;
			
			\node at (9.6,4.3) {$\textcolor{gray}{\Gamma}$};
			\draw[color=green] (-1.5,3.4) -- (10,3.4);
			\draw[color=green] (-1.5,2.4) -- (10,2.4);
			\draw[color=green, line width=2] (0,2.4) -- (0,3.4);
			\fill[fill=green!10, opacity=0.2] (-1.5,3.4) -- (10,3.4) -- (10,2.4) -- (-1.5,2.4) -- cycle;
			\filldraw [green] 	(-0.4142-1.4142/2, 2.4142+1.4142/2) circle (1pt)
			(-0.4142, 2.4142) circle (1pt)
			(-0.4142+1/2, 2.4142+1/2) circle (1pt)
			( 2.5858/2, 5.4142/2) circle (1pt)
			( 3.5858/2, 6.4142/2) circle (1pt)
			(5/2,5/2) circle (1pt)
			(3,3) circle (1pt)
			(8.4142/2, 5.5858/2) circle (1pt)
			(9.4142/2, 6.5858/2) circle (1pt)
			(5.4142, 2.5858) circle (1pt)
			(5.4142+1/2, 2.5858+1/2) circle (1pt)
			(4.8284-1.4142/2+6/2,-0.8284+1.4142/2+6/2) circle (1pt)
			(4.8284-1.4142/2+7/2,-0.8284+1.4142/2+7/2) circle (1pt)
			(7.8284+1/2, 2.1716+1/2) circle (1pt)
			(8.8284, 3.1716) circle (1pt)
			;		
			\node at (0.3,2.6) {$\textcolor{green}{W}$};			
			\filldraw[blue]  	(-0.4142-1.4142/2, 0) circle (1pt)
			(-0.4142, 0) circle (1pt)
			(-0.4142+1/2, 0) circle (1pt)
			( 2.5858/2,0) circle (1pt)
			( 3.5858/2, 0) circle (1pt)
			(5/2,0) circle (1pt)
			(3,0) circle (1pt)
			(8.4142/2, 0) circle (1pt)
			(9.4142/2, 0) circle (1pt)
			(5.4142, 0) circle (1pt)
			(5.4142+1/2, 0) circle (1pt)
			(4.8284-1.4142/2+6/2,0) circle (1pt)
			(4.8284-1.4142/2+7/2,0) circle (1pt)
			(7.8284+1/2, 0) circle (1pt)
			(8.8284,0) circle (1pt)
			;
			\node at (4.2,-0.3) {$\textcolor{blue}{\Lambda}$};
			\draw[color=red] (-1.5,1) -- (1.5,1);
			\draw[color=red] (-1.5,1) -- (-1.5,-1);
			\draw[color=red] (1.5,-1) -- (1.5,1);
			\draw[color=red] (1.5,-1) -- (-1.5,-1);
			\fill[fill=red!10, opacity=0.2] (-1.5,-1) -- (-1.5,1) -- (1.5,1) -- (1.5,-1) -- cycle;
			\filldraw[red] 		(-1,-1) circle (1pt)
			(-1/2,-1/2) circle (1pt)
			(0,0) circle (1pt)
			(1/2,1/2) circle (1pt)
			(1,1) circle (1pt)
			(-2.4142/2, 0.4142/2) circle (1pt)
			(-1.4142/2, 1.4142/2) circle (1pt)
			(1.4142/2,-1.4142/2) circle (1pt)
			(2.4142/2,-0.4142/2) circle (1pt)
			;
			\node at (0.3,-1.6) {\textcolor{red}{$r$-slab}};
		\end{tikzpicture}
		\caption{This figure shows the preimage of the slab for a fixed $r$ in the setting of a $\RR \times \RR$ model set.}
		\label{Fig:Slab}
	\end{figure}
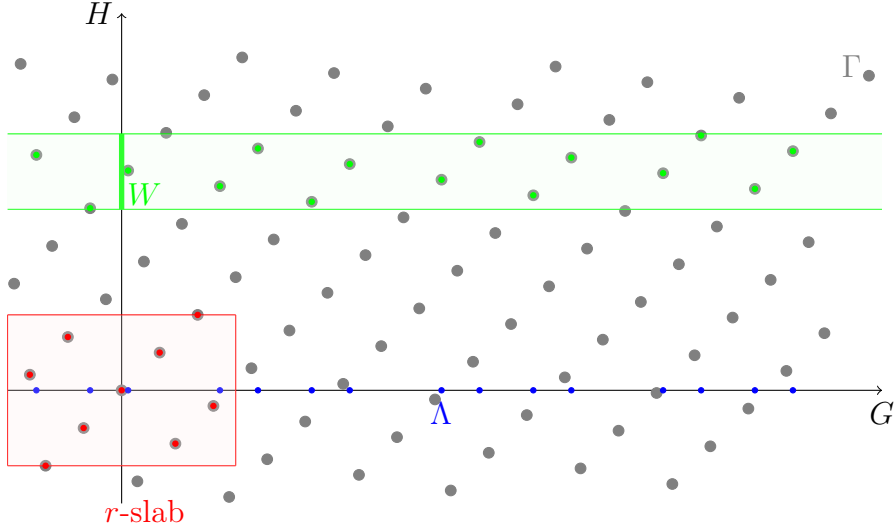\noindent

	\vspace{-1cm}
	\begin{Lemma}\label{Lem:Equiv}
		For $\lambda, \mu \in \Lambda$ we have that: $\lambda \sim_r \mu \Leftrightarrow \cS_r(\lambda)=\cS_r(\mu)$.
		\begin{proof}
			Assume $\lambda \sim_r \mu$, then $(B_r(\lambda) \cap \Lambda)\lambda^{-1}=(B_r(\mu) \cap \Lambda)\mu^{-1}$. Let $x \in S_r(\lambda)$ then there exists a $(\gamma, x)\in \Gamma$ such that
			\begin{align*}
				\gamma^{-1}\lambda &\in B_r(\lambda) \cap \Lambda\\
				&\Leftrightarrow \gamma^{-1} \in (B_r(\lambda) \cap \Lambda)\lambda^{-1} = (B_r(\mu) \cap \Lambda)\mu^{-1}\\
				&\Leftrightarrow \gamma^{-1}\mu \in B_r(\mu) \cap \Lambda.
			\end{align*}
			Therefore $x \in S_r(\mu)$.\par 
			
			Now assume $\cS_r(\lambda)=\cS_r(\mu)$ and let $x \in (B_r(\lambda)\cap \Lambda)\lambda^{-1}$ then $\tau(x^{-1}) \in \cS_r(\lambda)=\cS_r(\mu)$ and this implies $x \in (B_r(\mu)\cap \Lambda)\mu^{-1}$.
		\end{proof}
	\end{Lemma}
	
	\begin{proof}[Proof of  \cref{Thm:AcceptanceDomains}]
		\Cref{Lem:Equiv} tells us that for all $\lambda \in A_r^G(\lambda)$ the set
		\begin{equation*}
			W_r(\lambda):=\left(\bigcap_{\mu \in \cS_r(\lambda)}\mu \mathring{W}\right)\cap \left(\bigcap_{\mu \in \cS_r(\lambda)^\mathrm{C}} \mu W^\mathrm{C} \right)
		\end{equation*}
		is the same. So to prove \cref{Thm:AcceptanceDomainsEq1} it is enough to show $\tau(\lambda)\in W_r(\lambda)$. By the definition of the $r$-slab of $\lambda$ we have for all $\mu \in \cS_r(\lambda)$ that there is a $\mu_G \in \Gamma_G$ with $\tau(\mu_G)=\mu$ and $\mu_G^{-1}\lambda \in \Lambda$. Further \cref{Lem:WindowShift} tells us that $\tau(\lambda) \in \mu \mathring{W}$. For $\mu \in \cS_r^\mathrm{C}(\lambda)$ \cref{Lem:WindowShift} tells us that $\tau(\lambda) \not\in \mu W$, but this means that $\tau(\lambda) \in \mu W^\mathrm{C}$. So it follows that $\tau(\lambda)\in W_r(\lambda)$.\par \medskip
		Now let $\lambda \not\sim_r \lambda'$, so by \cref{Lem:Equiv} $S_r(\lambda) \neq S_r(\mu)$ and the disjointness of $W_r(\lambda)$ and $W_r(\lambda')$ follows by the same argument.\par\medskip
		
		Finally we show that the $\overline{W_r(\lambda)}$ tile the closure of the window $\overline{W}$. The inclusion ${\overline{W_r(\lambda)} \subseteq \overline{W}}$ is clear since $e_H \in \cS_r(\lambda)$ for all $\lambda \in \Lambda$ and all $r>0$. Since $\Gamma_H$ is dense in $W$ and $W_r(\lambda)$ is open we know that $\Gamma_H$ is dense in $W_r(\lambda)$. Since $A_r^H(\lambda)=\Gamma_H \cap W_r(\lambda)$ we know that $A_r^H(\lambda)$ is dense in $W_r(\lambda)$. Therefore the completion by sequences $\overline{A_r^H(\lambda)}^{seq}$ is the topological closure $\overline{W_r(\lambda)}^{top}$. Further since every $\gamma \in \Gamma_H \cap W$ has to belong to some $W_r(\lambda)$ we get that
		\begin{equation*}
			W \cap \Gamma_H= \bigcup_{\lambda \in A_r^G} A_r^H(\lambda).
		\end{equation*}
		Completion by sequences on both sides delivers
		\begin{equation*}
			\overline{W} = \bigcup_{\lambda \in A_r^G} \overline{A_r^H(\lambda)}^{seq} = \bigcup_{\lambda \in A_r^G} \overline{W_r(\lambda)}^{top}.
		\end{equation*}
	\end{proof}
	
	\begin{Remark}
		Taking the closure of the window in the theorem does not make a big difference, since by $\Gamma$-regularity there are no projected lattice points on the boundary. This also holds for the shifted window, since if $\gamma_1, \gamma_2 \in \Gamma_H$ with $\gamma_1 \in \gamma_2 W$, then $\gamma_2^{-1}\gamma_1 \in W$ in contradiction to the $\Gamma$-regularity of $W$. So for all acceptance domains $\partial W_{r}(\lambda) \cap \Gamma_H = \emptyset$.
	\end{Remark}
	
	\section{Lattice point counting}\label{Sec:Lattice}
	Before we begin with the actual proof of our main theorem we will establish the growth lemma, \cref{Prop:Growth}. The growth lemma tells us how to count points in sets of the form $(B_r(e) \times A) \cap \Gamma$. Notice that in particular the slab $\cS_r=\pi_H\left(\left(B_r(e)\times WW^{-1}\right)\cap \Gamma\right)$ is of this form.
	
	\begin{Proposition}[Growth Lemma]\label{Prop:Growth}
		Let $G$ and $H$ be lcsc groups. For a model set with a uniform lattice $\Gamma \subset G \times H$ and a bounded open set ${\emptyset \neq A \subset H}$. The asymptotic growth of the number of lattice points inside $B^G_r(e) \times A$ is bounded by
		\begin{equation*} 
			\mu_G\left(B^G_{r-k_1}(e)\right) \ll \left\vert\left(B^G_r(e)\times A\right)\cap \Gamma \right\vert \ll \mu_G\left(B^G_{r+k_2}(e)\right),
		\end{equation*}
		for some constants $k_1,k_2>0$ as $r \to \infty$.
	\end{Proposition}
	
	The proof consists of the following well known proposition, see for example \cite[Lemma 7.4]{BaakeGrimm}, and the two subsequent lemmas.
	
	\begin{Proposition}\label{Lem:ProduktOffenKompakt}
		Let $G$ and $H$ be lcsc groups and $\Gamma \subset G \times H$ a uniform lattice, such that $\pi_H(\Gamma)$ is dense in $H$. Further let $U \subset H$ be an open non-empty set. Then there exists a compact set $K \subset G$ such that
		\begin{equation*}
			G \times H = (K \times U)\Gamma.
		\end{equation*}
		\begin{proof}
			Since $\Gamma$ is a uniform lattice in $G \times H$ there exists a compact set $C$ such that $G \times H= C\Gamma$. We can cover $C$ by the bigger compact set $\pi_G(C) \times \pi_H(C)$, which implies $ G \times H = (C_G \times C_H)\Gamma$. By density of $\pi_H(\Gamma)$ in $H$ we get a covering
			\begin{equation*}
				\bigcup_{ \gamma \in \Gamma} U\pi_H(\gamma) = H \supset C_H.
			\end{equation*}
			Since $C_H$ is compact we can choose a finite subcovering with finite $F\subset \Gamma$ such that
			\begin{equation*}
				\bigcup_{ \gamma \in F} U\pi_H(\gamma) \supset C_H.
			\end{equation*}
			Now let $z\in  G \times H$ be arbitrary. By the choice of $C$ we find a $\gamma \in \Gamma$ such that ${z \gamma^{-1} \in C \subset C_G \times C_H}$. By our covering argument we find a $f \in F$ such that ${\pi_H(z \gamma^{-1}) \in U \pi_H(f)}$ and therefore $\pi_H(z \gamma^{-1} f^{-1}) \in U$. If we project the same element to $G$ we get
			\begin{equation*}
				\pi_G(z \gamma^{-1} f^{-1}) \in C_G \pi_G(F^{-1}) =:K.
			\end{equation*}
			Now $K$ is compact since $C_G$ is compact and $\pi_G(F^{-1})$ is finite. Putting things together we realise
			\begin{equation*}
				z = (z \gamma^{-1} f^{-1}) (f \gamma) \in (K \times U) \Gamma.
			\end{equation*}
		\end{proof}
	\end{Proposition}
	
	\begin{Lemma}\label{Lem:GrowthSrUpperGeneral}
		Let $G$ and $H$ be  lcsc groups. For a model set with a uniform lattice $\Gamma \subset G \times H$ and a bounded open set ${\emptyset \neq A \subset H}$ the growth of the lattice points inside $B^G_r(e) \times A$ is asymptotically bounded from above by
		\begin{equation*} 
			\left\vert(B^G_r(e)\times A)\cap \Gamma \right\vert \ll \mu_G(B^G_{r+k_2}(e)),
		\end{equation*}
		where $k_2$ is some constant as $r \to \infty$.
		\begin{proof}
			Since $\Gamma$ is a lattice it is uniformly discrete, therefore we find a constant $c_1$ such that $d(\gamma_1,\gamma_2)> c_1$ for all $\gamma_1 \neq \gamma_2 \in \Gamma$. If we halve the constant we get disjoint balls around the lattice points, i.e. $B^{G\times H}_{\frac{c_1}{2}}(\gamma_1) \cap B^{G\times H}_{\frac{c_1}{2}}(\gamma_2)=\emptyset$. \par\medskip
			Since $A$ is bounded we find a second constant $c_2$ such that $A \subset B_{c_2}^H(e)$ and that $B_{c_1}^{G\times H}(x)\subset G \times B^H_{c_2}(e)$ for every $x \in G\times A$. The norm in the product is given by the maximum of the norms of the components.\par
			
			The idea is that we build a set which contains not only the points of $(B^G_r(e)\times A)\cap \Gamma$, but also the balls around them. Then we can obtain an upper bound for the number of points in $(B^G_r(e)\times A) \cap \Gamma$ by estimating how often the thickened set of points could fit in this set via a volume estimate.
			Since by our choice of $\frac{c_1}{2}$ the balls do not overlap, we obtain that
			\begin{equation*}
				\sum_{\gamma \in (B^G_r(e)\times A) \cap \Gamma} \mu_{G \times H}\left(B^{G \times H}_{\frac{c_1}{2}}(\gamma)\right) \leq \mu_{G\times H}\left(B^G_{r+c_1}(e) \times B^H_{c_2}(e)\right). 
			\end{equation*}
			We need the ''$+c_1$`` in the index so the set can also contain all the balls whose center lie close to the border of $B^G_r(e)$. The volume of a ball is independent of its center point, since the metric and the Haar-measure are right-invariant. Therefore we can write the inequality as
			\begin{equation*}
				\left\vert\left(B^G_r(e)\times A\right) \cap \Gamma\right\vert \cdot \mu_{G \times H}\left(B^{G \times H}_{\frac{c_1}{2}}(e)\right) \leq \mu_{G \times H}\left(B^G_{r+c_1}(e)\times B^H_{c_2}(e)\right). 
			\end{equation*}
			We will now divide this equation by $\mu_{G \times H}\left(B^{G\times H}_{\frac{c_1}{2}}(e)\right)$, which is just a constant dependent on $c_1$ which we will denote by $c_1'$, and the constant $\mu_{H}\left(B^H_{c_2}(e)\right)$ will be denoted by $c_2'$.
			\begin{align*}
				\left\vert\left(B^G_r(e)\times A\right) \cap \Gamma\right\vert &\leq \frac{\mu_{G \times H}\left(B^G_{r+c_1}(e) \times B^H_{c_2}(e)\right)}{c_1'}\\ &=\frac{\mu_{G}\left(B^G_{r+c_1}(e)\right) \cdot \mu_{H}\left(B^H_{c_2}(e)\right)}{c_1'}
				= \frac{c_2'}{c_1'} \mu_G\left(B^G_{r+c_1}(e)\right). 
			\end{align*}
		\end{proof}
	\end{Lemma}
	
	\begin{Lemma}\label{Lem:GrowthSrLowerGeneral}
		Let $G$ and $H$ be lcsc groups. For a model set with a uniform lattice $\Gamma \subset G \times H$ and a bounded open set ${\emptyset \neq A \subset H}$ the growth of the number of lattice points inside $B^G_r(e) \times A$ is asymptotically bounded from below by
		\begin{equation*} 
			\left\vert\left(B^G_r(e)\times A\right)\cap \Gamma \right\vert \gg \mu_G\left(B^G_{r-k_1}(e)\right),		
		\end{equation*}
		where $k_1$ is some constant as $r \to \infty$.
		\begin{proof}
			Let $\epsilon>0$ be fixed. We choose an open ball $B^H_\epsilon(\gamma_H) \subset A$ with $\gamma_H \in \Gamma_H$, this can be done since $\Gamma_H$ is dense in $H$ and $A$ is open and therefore $\Gamma_H \cap A \subset A$ is dense.\par
			First we assume that $\gamma_H = e$. By \cref{Lem:ProduktOffenKompakt} we find a compact set $K \subset G$ such that $ G \times H = \left(K \times B^H_\epsilon(e)\right)\Gamma$. Since $K$ is compact it is bounded, we can consider $\overline{B^G_{c_1}}(e)$, with $c_1$ large enough, instead. Then for all $z \in G \times H$ we see $\left(B^G_{c_1}(e) \times B^H_\epsilon(e)\right) z \cap \Gamma \neq \emptyset$. This holds true since we can write $z=(k_z,u_z)(\gamma_{zG},\gamma_{zH})$ with $(\gamma_{zG},\gamma_{zH}) \in \Gamma$, $k_z \in B^G_{c_1}(e)$ and $u_z \in B^H_\epsilon(e)$. But then
			\begin{equation*}
				(\gamma_{zG},\gamma_{zH}) = \left(k_z^{-1},u_z^{-1}\right) z \in \left(B^G_{c_1}(e) \times B^H_\epsilon(e)\right) z \cap \Gamma,
			\end{equation*}
			since $k_z^{-1} \in B^G_{c_1}(e)$ and $u_z^{-1} \in B^H_\epsilon(e)$.\par 
			We can find a lower bound of the growth if we can fit enough of the sets of type ${\left(B^G_{c_1}(e)\times B^H_\epsilon(e)\right) z}$ into $B^G_r(e) \times A$ in a disjoint way. One should think of this as stacking these sets onto another with base $B^H_\epsilon(e)$.  This comes down to
			\begin{align*} 
				&\left\vert\left(B^G_r(e)\times A\right)\cap \Gamma \right\vert > \left\vert\left(B^G_r(e) \times B^H_\epsilon(e)\right)\cap \Gamma \right\vert\\ 
				&\quad\geq \max \left\{ \vert X \vert \,\middle\vert\, X \subset G,\text{ such that } \forall x \in X:  B^G_{c_1}(x) \subset B^G_r(e) \right.\\
				&\quad\quad \quad \quad \left.\text{ and } B^G_{c_1}(x)\cap B^G_{c_1}(y) = \emptyset\,\, \forall x \neq y \in X \right\}\\
				&\quad\geq \max \left\{ \vert X \vert \,\middle\vert\, X \subset B^G_{r-c_1}(e) \text{ and } X \text{ is } 2c_1\text{-uniformly discrete}\right\}.
			\end{align*}
			We can extend every $c_1$-uniformly discrete set to a $(c_2,c_1)$-Delone set for some constant $c_2$, \cite[Proposition 3.C.3]{Harpe}. Thus
			\small\begin{equation*}
				\left\vert\left(B^G_r(e)\times A\right)\cap \Gamma \right\vert > \max \left\{ \vert X \vert \,\middle\vert\, X \subset B^G_{r-c_1}(e) \text{ and } X \text{ is a } (c_2,2c_1)\text{-Delone subset of } B_r^G(e)\right\}.
			\end{equation*}\normalsize
			For every such Delone set we can cover $B^G_{r-c_1}(e)$ with balls $B^G_{c_2}(x)$ for $x\in X$, so that
			\begin{align*}
				\bigcup_{x \in X} B^G_{c_2}(x) \supset B^G_{r-c_1}(e) \Rightarrow \sum_{x\in X} \mu_{G}\left(B^G_{c_2}(x)\right) \geq \mu_G\left(B^G_{r-c_1}(e)\right).
			\end{align*}
			Since the metric and the Haar-measure are right-invariant all of these balls have the same measure and we get
			\begin{align*}
				\vert X \vert \cdot \mu_G\left(B^G_{c_2}(e)\right) \geq \mu_G\left(B^G_{r-c_1}(e)\right) \Leftrightarrow \vert X \vert \geq \frac{\mu_G\left(B^G_{r-c_1}(e)\right)}{\mu_G\left(B^G_{c_2}(e)\right)}.
			\end{align*}
			Summing up we have
			\begin{equation*}
				\left\vert\left(B^G_r(e)\times A\right)\cap \Gamma \right\vert > \frac{\mu_G\left(B^G_{r-c_1}(e)\right)}{\mu_G\left(B^G_{c_2}(e)\right)}.
			\end{equation*}
		\end{proof}
	\end{Lemma}	
	
	\begin{Definition}
		Let $G$ be a locally compact group and let $d$ be a right-invariant metric on $G$ compatible with the topology on $G$. Then $G$ is a group with \textit{exact polynomial growth of degree $\kappa$ with respect to $d$} if there exists a constant $c>0$ such that
		\begin{equation*}
			\lim\limits_{r \to \infty} \frac{\mu_{G}(B_r(e))}{c r^\kappa} =1.
		\end{equation*}
	\end{Definition}
	
	\begin{Corollary}
		If $G$ is a group with exact polynomial growth of degree $\kappa$ with respect to $d$, then
		\begin{equation*} 
			\left\vert\left(B^G_r(e)\times A\right)\cap \Gamma \right\vert \asymp r^\kappa.
		\end{equation*}
	\end{Corollary}
	
	\section{Homogeneous Lie groups}\label{Sec:HomLie}
	In this section we will review the basic concepts of homogeneous Lie groups, following the book by Fisher and Ruzhansky \cite{FischerRuzhansky}. In this context we also recall an ergodic theorem, due to Gorodnik and Nevo, \cite{GorodnikNevo1, GorodnikNevo2, NevoErgodic}, which will be recalled later.
	
	\begin{Definition}\textbf{\cite[Definition 3.1.7]{FischerRuzhansky}}\label{Def:HomLie}
		\begin{enumerate}
			\item A family of \textit{dilations} of a Lie algebra $\fg$ is a family of linear mappings $\{D_r, r>0\}$ from $\fg$ to itself which satisfies:
			\begin{enumerate}[label=\roman*)]
				\item The mappings are of the form
				\begin{equation*}
					D_r=\exp(\ln(r) A) = \sum_{l=0}^\infty \frac{1}{l !}(\ln(r)A)^l,
				\end{equation*}
				where $A$ is a diagonalisable linear operator on $\fg$ with positive eigenvalues, $\exp$ denotes the exponential of matrices and $\ln(r)$ the natural logarithm of $r>0$.
				\item Each $D_r$ is a morphism of the Lie algebra $\fg$, that is, a linear mapping from $\fg$ to itself which respects the Lie bracket, i.e.
				\begin{equation*}
					\forall X,Y \in \fg, r>0: [D_r X, D_r Y] = D_r[X,Y].
				\end{equation*}
			\end{enumerate}
			\item A \textit{homogeneous} Lie group is a connected simply connected Lie group whose Lie algebra is equipped with a fixed family of dilations.
			\item We call the eigenvalues of $A$ the \textit{dilations' weights}, the sum of these weights is the \textit{homogeneous dimension of $\fg$}, denoted by $\homdim(\fg)$.
		\end{enumerate}
	\end{Definition}
	
	\begin{Convention}
		From now on we will always assume $G$ and $H$ to be homogeneous Lie groups.
	\end{Convention}
	
	\begin{Remark}
		Since every Lie algebra, which is equipped with dilations, is nilpotent, a homogeneous Lie group is nilpotent, \cite[Proposition 3.1.10.]{FischerRuzhansky}.  This together with the connectedness and the simply connectedness implies that the exponential map is a global diffeomorphism, see \cite[Proposition 1.6.6.]{FischerRuzhansky}, which we use to identify the underlying sets of $G$ and $\fg$. On $\fg$ the group multiplication takes the form
		\begin{equation*}
			X \ast Y := \log(\exp(X)\exp(Y)).
		\end{equation*}
		The operation $\ast$ is called the Baker-Campbell-Hausdorff (BCH) multiplication, since to determine it we can use the BCH formula. The explicit formula for multiplication then is
		\begin{equation*}
			X \ast Y = X + \sum_{\substack{k,m \geq 0 \\ p_i+q_i \geq 0\\ i \in \{0,...,k\}}} (-1)^k \frac{{\ad_X}^{p_1} \circ {\ad_Y}^{q_1} \circ ... \circ {\ad_X}^{p_k} \circ {\ad_Y}^{q_k} \circ {\ad_X}^m }{(k+1)(q_1+...+q_k+1) \cdot p_1!\cdot q_1! \cdot ... \cdot p_k! \cdot q_k! \cdot m!}(Y).
		\end{equation*}
		In the case of a $n$-step nilpotent Lie group this sum is finite, since all terms there $m+\sum_i p_i + q_i \geq n$ are zero. The first few terms of the sum then look like 
		\begin{equation*}
			X \ast Y = X + Y +\frac{1}{2}[X,Y] + \frac{1}{12} \big([X,[X,Y]]-[Y,[X,Y]]\big)- \frac{1}{24} [Y,[X,[X,Y]]] + ...
		\end{equation*}
		Observe that the inverse of $X$ in terms of the action $\ast$ is given by the additive inverse of $X$ in the Lie algebra, i.e. $X^{-1}= -X$.\par
		
		In particular the group law is a polynomial, see \cite[Proposition 1.6.6.]{FischerRuzhansky}, this means that for $x=(x_1,...,x_n)$ and $y=(y_1,...,y_n)$ we have
		\begin{equation*}
			x \ast y = (P_1(x,y),..., P_n(x,y)),
		\end{equation*} 
		with $P_1,..., P_n$ polynomials in $2n$ variables, we will observe some restrictions for this polynomials in \cref{Appendix:Poly2}.\par \medskip 
		
		Observe that we can transport the dilations of the Lie algebra to the Lie group via the exponential map, so that we also have a dilation structure on the Lie group itself. This results in a family of dilations of the form $D_r=\exp(A \ln(r))$ with a diagonalisable linear operator $A$. The eigenvalues of $A$ are the weights we mentioned above, in accordance with \cite{FischerRuzhansky} we denote this weights by $\nu_1,...,\nu_n$. The trace of $A$ gives us the \textit{homogeneous dimension of $G$}, the properties listed below explain this terminology.\par \medskip
	\end{Remark}
	
	\begin{Definition}[{\cite[Definition 3.1.33.]{FischerRuzhansky}}]
		A \textit{homogeneous quasi-norm} is a continuous non-negative function $G \to [0,\infty), x \mapsto \vert x \vert$ satisfying
		\begin{enumerate}[label=\roman*)]
			\item $\vert x^{-1}\vert = \vert x \vert$,
			\item $\vert D_r( x )\vert = r \cdot \vert x \vert$ $ \forall r>0$,
			\item $\vert x \vert =0$ if and only if $x=e$.
		\end{enumerate}
		It is called a \textit{homogenous norm} if additionally
		\begin{enumerate}[label=\roman*)]
			\setcounter{enumi}{3}
			\item for all $x,y \in G$ it is $\vert xy\vert \leq \vert x \vert + \vert y\vert$.
		\end{enumerate}
	\end{Definition}
	
	\begin{Lemma}[{\cite[Theorem 3.1.39.]{FischerRuzhansky}}]\label{Lem:HomNorm}
		If $G$ is a homogeneous Lie group, then there exists a homogeneous norm $\vert \cdot \vert$ on $G$.
	\end{Lemma}
	
	\begin{Lemma}[{\cite[Proposition 3.1.35.]{FischerRuzhansky}}]
		Any two homogeneous quasi-norms $\vert \cdot \vert$ and $\vert \cdot \vert'$ on $G$ are mutually equivalent, in the sense that there exists $a,b >0$ such that for all $g \in G$ it is $a \vert x \vert' \leq \vert x\vert \leq b \vert x\vert'$.
	\end{Lemma}
	
	\begin{Remark}
		To such a homogeneous norm we can associate the right invariant metric $d$ given by $d(x,y):= \vert x y^{-1}\vert$.
	\end{Remark}
	
	\begin{Lemma}[{\cite[Proposition 3.1.37.]{FischerRuzhansky}}]\label{Lem:HomLieTop}
		If $\vert \cdot \vert$ is a homogeneous quasi-norm on a homogeneous Lie group $G$ of dimension $n$, then the topology induced by the quasi-norm coincides with the Euclidean topology on the underlying set $\RR^n$.
	\end{Lemma}
	
	\begin{Proposition}[{\cite[Section 3.1.3 and 3.1.6]{FischerRuzhansky}}]\label{Prop:BallMeas}
		Let $G$ be a homogeneous Lie group and $\vert \cdot \vert$ a homogeneous norm on $G$ with associated right invariant metric $d$, then for $x,y \in G$ and $r,s >0$:
		\begin{enumerate}[label=\roman*)]
			\item $B_r(x) = \{y \in G \mid \vert y x^{-1}\vert < r\} = B_r(e) \cdot x$,
			\item $D_r(x y) = D_r(x) D_r(y)$,
			\item $D_r(B_s(e)) = B_{r\cdot s}(e)$,
			\item $D_r(B_s(x))=B_{r \cdot s}(D_r(x))$,
			\item $\mu_{G}(B_r(x))= r^{\homdim(G)}\cdot \mu_{G}(B_1(e))$.
		\end{enumerate}
	\end{Proposition} 
	
	\begin{Remark}
		The fifth point of \cref{Prop:BallMeas} tells us that a homogeneous Lie group has exact polynomial growth of degree $\homdim(G)$. Since all homogeneous quasi-norms on $G$ are mutually equivalent the behaviour is independent from the choice of a metric.
	\end{Remark}
	
	\begin{Definition}
		A \textit{hyperplane} in a homogeneous Lie group $G$ is the image of a hyperplane in the Lie algebra $\fg$ under the exponential map. The set of all hyperplanes in $G$ is denoted by $\cH(G)$.\par
		A \textit{half-space} in a homogeneous Lie group $G$ is the image of a half-space in the Lie algebra $\fg$ under the exponential map. 
	\end{Definition}
	
	\begin{Definition}
		A group $G$ is \textit{non-crooked} if $\cH(G) \subset \cP(G)$ is $G$ invariant.
	\end{Definition}
	
	\begin{Definition}
		A Lie group $G$ is called \textit{locally $k$-step nilpotent} if for all $X,Y \in \fg$ we have $\ad_X^k(Y) = 0$.
	\end{Definition}
	
	\begin{Theorem}\label{Thm:NonCrooked}
		Let $G$ be a homogenous Lie group, then the following are equivalent
		\begin{enumerate}
			\item $G$ is non-crooked,
			\item $G$ is $2$-step nilpotent or abelian,
			\item $G$ is locally $2$-step nilpotent.
		\end{enumerate}
	\end{Theorem}
	
	For a proof of the theorem see \cref{Appendix:Poly}.
	\begin{Remark}
		The notion of locally $k$-step nilpotent and $k$-step nilpotent are only equivalent for $k=1$ and $k=2$ for greater $k$ the notions are different. 
	\end{Remark}
	
	\begin{Definition}
		We call a window \textit{polytopal} or of \textit{polytopal type} if it is the intersection of finitely many half-spaces.
	\end{Definition} 
	
	\begin{Convention}
		If $W$ is polytopal, we can express $W$ as $\bigcap_{i=1}^N P_i^+$, where each $P_i^+$ is a half-space with opposite half-space $P_i^-$ and bounding hyperplane $P_i$.
		Further we denote the faces of $W$ by $\partial_i W=W \cap P_i$. We will in the rest of the paper always use this notation for the half spaces and hyperplanes associated to a window of polytopal type.
	\end{Convention}
	
	We give a small example which the reader can keep in mind, we will now fix the basics of this example so the reader can follow the steps in the paper. 
	\begin{Example}
		We set $G=H=\HH$, where $\HH$ denotes the Heisenberg group. We view the underlying set of $\HH$ as $\RR^3$. Further we set 
		\begin{equation*} 
			\Gamma= \left\{(a,b,c, a^\ast,b^\ast,c^\ast) \in G \times H \,\middle\vert\, a,b,c \in \ZZ\left[\sqrt{2}\right]\right\}
		\end{equation*}
		and $W=[-\frac{1}{2},\frac{1}{2}]\times[-\frac{1}{2},\frac{1}{2}]\times[-\frac{1}{2},\frac{1}{2}]$, so it is clear that $W$ is non-empty, pre-compact and $\Gamma$-regular. Also $W$ is a polytope, in fact it is a cube.\par 
		A dilation structure on $\HH$ is given by $D_r((x,y,z)) = \left(r \cdot x, r \cdot y , r^2 \cdot z\right)$, further a homogeneous norm is given by the Kor{\'a}nyi-Cygan norm $\vert (x,y,z)\vert _{\HH} = \left( (x^2+y^2)^2 + z^2 \right)^{\frac{1}{4}}$.
	\end{Example}

	\subsection{Ergodic theorems for homogeneous Lie groups}\label{SubSec:Ergodic}
	
	\begin{Definition}[{\cite[Definition 1.1.]{GorodnikNevo2}}]
		Let $O_\epsilon$, $\epsilon >0$ be a family of symmetric neighbourhoods of the identity in a lcsc group $G$, which are decreasing in $\epsilon$. Then a family of bounded Borel subsets of finite Haar-measure $(B_t)_{t>0}$ is \textit{well-rounded w.r.t. $O_\epsilon$} if for every $\delta >0$ there exists $\epsilon, t_1 >0$ such that for all $t \geq t_1$ 
		\begin{equation*}
			\mu_G(O_\epsilon B_t O_\epsilon) \leq (1+\delta) \mu_G\left(\bigcap_{u,v \in O_\epsilon}u B_t v\right).
		\end{equation*}
	\end{Definition}
	In our setup we will always fix $O_\epsilon$ as $B_\epsilon(e)$, this does not make a difference by \cite[Remark 2.3.]{Yakov}.
	
	\begin{Definition}[{\cite[Definition 1.4 and 1.5]{GorodnikNevo2}}]
		Let $G$ be a lcsc group and $B_t$ a family of bounded Borel subsets of finite Haar-measure. And let $\beta_{G,B_t}$ be the operator
		\begin{equation*}
			\beta_{G,B_t}f(x) := \frac{1}{\mu_{G}(B_t)} \int_{B_t} f(g^{-1}x) d\mu_G(g)
		\end{equation*}
		for $f \in L^2(G)$. We say that the \textit{mean ergodic theorem in $L^2(G)$} holds if
		\begin{equation*}
			\left\vert\left\vert \beta_{G,B_t}f - \int_G f d\mu \right\vert\right\vert_{L^2(G)} \to 0 \text{ as } t \to \infty
		\end{equation*}
		for all $f \in L^2(G)$.
		We say that the \textit{stable mean ergodic theorem in $L^2(X)$} holds if the \textit{mean ergodic theorem in $L^2(G)$} holds for the sets
		\begin{equation*}
			B_t^+(\epsilon) = O_\epsilon B_t O_\epsilon \text{ and } B_t^-(\epsilon)=\bigcap_{u,v \in O_\epsilon}u B_t v,
		\end{equation*}
		for all $\epsilon \in (0, \epsilon_1)$ with $\epsilon_1 >0.$
	\end{Definition}
	
	\begin{Remark}
		From now on we fix a Haar-measure $\mu_{G \times H}$, which we assume to be normalized by $\mu_{G \times H /\Gamma}(G \times H / \Gamma) = 1$.
	\end{Remark}
	
	\begin{Theorem}[{\cite[Theorem 1.7]{GorodnikNevo2}}]\label{Thm:Nevo}
		Let $G$ be a lcsc group, $\Gamma \subset G$ a discrete lattice subgroup and $(B_t)_{t>0}$ a well-rounded family of subsets of $G$. Assume that the averages $\beta_{G/\Gamma, B_t}$ supported on $B_t$ satisfy the stable mean ergodic theorem in $L^2(G/\Gamma)$. Then
		\begin{equation*}
			\lim\limits_{t \to \infty} \frac{\vert \Gamma \cap B_t\vert}{\mu_{G}(B_t)} = 1.
		\end{equation*}
	\end{Theorem}
	
	To apply this theorem, we have to show that the sets we consider are well-rounded and that they satisfy the stable means ergodic theorem. We will give criteria which ensure this.
	
	\begin{Lemma}\label{Lem:WellRoundedBalls}
		Let $G$ be a homogeneous Lie group and $(B_t(x))_{t>0}$ a family of balls in $G$. Then this family is well-rounded.
		\begin{proof}
			We have to show that for every $\delta>0$ there exists some $\epsilon, t_1 >0$ such that for all $t \geq t_1$ holds
			\begin{equation*}
				\mu_G(B_\epsilon(e) B_t(x) B_\epsilon(e)) \leq (1+\delta) \mu_G\left(\bigcap_{u,v \in B_\epsilon(e)}u B_t(x) v\right).
			\end{equation*}
			We first show that we can choose $\epsilon$ such that for a constant $k \in (0,t)$ we have $B_{t-k}(x) \subset \bigcap_{u,v \in B_\epsilon(e)}u B_t(x) v$. So let $g \in B_{t-k}(x)$, then we can write $g$ as $u u^{-1} g v^{-1} v$ with $u,v \in B_\epsilon(e)$ and we have to show that $u^{-1} g v^{-1} \in B_t(x)$.
			\begin{align*}
				d(x, u^{-1}g v^{-1}) &= \vert x v g^{-1} u\vert_G = \vert x v x^{-1} x g^{-1} u\vert_G\\
				&\leq \vert x v x^{-1}\vert_G + \vert xg^{-1} \vert_G + \vert u\vert_G \leq c_x(\epsilon) + t-k + \epsilon.
			\end{align*}
			Here the last inequality holds by \cref{Cor:ConBall}. And we have to choose $\epsilon$ so small that $k > \epsilon + c_x(\epsilon)$, which is possible since $c(\epsilon) \to 0$ as $\epsilon \to 0$.\par
			On the other hand $B_\epsilon(e)B_t(x)B_\epsilon(e) \subset B_{\epsilon + t}(x)B_\epsilon(e)$. Let $y \in B_{t+\epsilon}(x)$ and $u \in B_\epsilon(e)$ then
			\begin{align*}
				d(x, yu) &= \vert x u^{-1} y^{-1}\vert_G = \vert x u^{-1} x^{-1} x y^{-1} \vert_G\\
				&\leq \vert x u ^{-1} x^{-1}\vert_G + \vert x y^{-1} \vert_G \leq c_x(\epsilon) + \epsilon + t.
			\end{align*}
			Therefore  $B_\epsilon(e)B_t(x)B_\epsilon(e) \subset B_{\epsilon + t + c_x(\epsilon)}(x)$. Therefore we can choose $\epsilon > 0$ and $k \in (0, t)$ such that 
			\begin{equation*}
				B_\epsilon(e) B_t(x) B_\epsilon(e) \subset B_{\epsilon + t + c_x(\epsilon)}(x) \text{ and } B_{t-k}(x) \subset \bigcap_{u,v \in B_\epsilon(e)}u B_t(x) v
			\end{equation*}
			hold simultaneously. Now we can use that we can calculate the measure of balls in homogeneous Lie groups by \cref{Prop:BallMeas}:
			\begin{equation*}
				\mu_G(B_\epsilon(e) B_t(x) B_\epsilon(e)) \leq \mu_G(B_{\epsilon + t + c_x(\epsilon)}(x)) = (t+\epsilon+c_x(\epsilon))^{\homdim(G)} \mu_G(B_1(e))
			\end{equation*}
			and 
			\begin{equation*}
				(t-k)^{\homdim(G)} \mu_G(B_1(e)) \leq \mu_G(B_{t-k}(x)) \leq \mu_G \left(\bigcap_{u,v \in B_\epsilon(e)}u B_t(x) v\right).
			\end{equation*}
			Combining the arguments we see that we have to choose $\epsilon$, $k$ and $t_1$ such that for all $t > t_1$
			\begin{equation*}
				\left(\frac{(t+\epsilon+c_x(\epsilon))}{(t-k)}\right)^{\homdim(G)} \leq (1+\delta).
			\end{equation*}
		\end{proof}
	\end{Lemma}
	
	\begin{Lemma}
		Let $G$ be a homogeneous Lie group and $(B_r(x))_{t>0}$ a constant family of balls in $G$. Then this family is well-rounded.
		\begin{proof}
			We have already seen in the proof of \cref{Lem:WellRoundedBalls} that $\bigcap_{u,v \in B_\epsilon(e)} u B_r(x) v$ contains a ball of the form $B_{r-k}(x)$ for any $k\in (0,t)$ if we choose $\epsilon$ accordingly. On the other hand  $B_\epsilon(e)B_r(x)B_\epsilon(e)$ is contained in a ball $B_{r+\epsilon+c_x(\epsilon)}(x)$ and therefore has finite measure. Choosing $\epsilon$ accordingly we are done.
		\end{proof}
	\end{Lemma}
	
	Now what is left to show is that the stable mean ergodic theorem holds for our families of sets. To do so we use \cite[Theorem 3.33.]{Glasner} which tells us that we have to check that our families are F\o lner sequences. Alternatively see the work by Nevo, \cite{NevoErgodic}, especially step I in the proof of Theorem 5.1 is exactly what we need.
	
	\begin{Definition}\textbf{(F\o lner sequences)}
		Let $G$ be a lcsc group acting on a measure space $(X,\mu)$. A sequence $F_1, F_2, ...$ of subsets of finite, non-zero measure in $X$ is called a \textit{(right) F\o lner sequence} if for all $g \in G$
		\begin{equation*}
			\lim\limits_{i \to \infty}\frac{\mu(F_i g \triangle F_i)}{\mu(F_i)} = 0.
		\end{equation*}
	\end{Definition}
	
	\begin{Lemma}
		Let $G \times H$ be a product of homogeneous Lie groups, $\Gamma \subset G \times H$ a lattice, $(B_t^G(e))_{t>0}$ a family of balls in $G$ and $(B_r^H(x))_{t>0}$ a constant family of balls in $H$. The family $(B_t^G(e) \times B_r^H(x))_{t>0}/\Gamma$ is a F\o lner sequence in $(G\times H)/ \Gamma$.
		\begin{proof}
			We use \cref{Lem:ProduktOffenKompakt}, which tells us that for every $r>0$ and every $x \in H$ we find a compact $K \subset G$ such that $(K \times B_r^H(x))\Gamma = G \times H$. And since every compact set $K$ is contained in some $B_t^G(e)$, for $t$ large enough, we get that $(B_t^G(e) \times B_r^H(x)) \Gamma = G \times H$. This also holds if we consider balls $B_t^G(g)$ instead of $B_t^G(e)$. Therefore we have for every $(g,h) \in G \times H$ that there exists a $t_{g,h} >0$ such that $(B_t^G(g) \times B_r^H(xh)) \Gamma = G \times H$ and so 
			\begin{equation*}
				\lim\limits_{t \to \infty} \mu_{(G \times H)/\Gamma}((B_t^G(g) \times B_r^H(xh))_{(G\times H)/\Gamma}  \triangle (B_t^G(e) \times B_r^H(x))_{(G\times H)/ \Gamma}) = 0
			\end{equation*}
			for all $(g,h) \in G \times H$.
		\end{proof}
	\end{Lemma}
	
	\section{Growth of the complexity function}\label{Sec:Growth}
	In this section we proof our main theorem, modulo \cref{Thm:BeckInduction} which we will proof in \cref{Sec:Combinatorics}.
	
	\begin{Theorem}\label{MainTHM}
		Let $\Lambda(G,H, \Gamma, W)$ be a model set, where $G$ be a homogeneous Lie group and $H$ a non-crooked homogeneous Lie group. Further $W \subset H$ be a non-empty, precompact window of polytopal type with bounding hyperplanes $P_1,..., P_N$ such that $P_i$ has trivial stabilizer and that $P_i \cap \Gamma_H = \emptyset$ for all $i \in \{1,...,N\}$. Then for the complexity function $p$ of $\Lambda$ we have
		\begin{equation*}
			p(r) \asymp r^{\homdim(G)\dim(H)}.
		\end{equation*} 
	\end{Theorem}
	
	Notice that the condition of $\Gamma$-regularity is hidden in the stronger condition that all the hyperplanes $P_i$ do not intersect $\Gamma_H$. The condition that the $P_i$ have trivial stabilizer simplifies the problem, since otherwise we had to address the effects of the stabilizer. The effect of a non trivial stabilizer is similar to the Euclidean case, for which we refer to \cite{KoivusaloWalton2}.\par \medskip
	
	The proof of the \cref{MainTHM} will be divided into establishing an upper bound and a lower bound. For the lower bound we have to put in a lot more effort.
	
	\subsection{Upper bound of the growth}\label{SubSec:UpperBound}
	For an upper bound we consider the decomposition of the window $W \setminus \bigcup_{x\in\cS(r)} x\partial W$. We will show that counting the connected components is an upper bound for the number of acceptance domains. Then we will use the theory of real hyperplane arrangements which gives us an upper bound for our counting.
	
	\begin{Lemma}
		For a model set $\Lambda(G,H,\Gamma,W)$ we have
		\begin{equation*}
			\left\vert A_r^H\right\vert \leq \#\pi_0 \left( W \setminus \bigcup_{\mu\in\cS_r} \mu\partial W\right).
		\end{equation*}
		\begin{proof}
			By \cref{Thm:AcceptanceDomains} we know that the acceptance domains $W_r(\lambda)$ tile the window $W$ and that they are disjoint. Further for every $A_r(\lambda)$ we know that
			\begin{equation*}
				A_r(\lambda) \subset W_r(\lambda).
			\end{equation*}
			So 
			\begin{align*}
				\partial W_r(\lambda) = &\partial \left( \left(\bigcap_{\mu \in \cS_r(\lambda)}\mu \mathring{W}\right)\cap \left(\bigcap_{\mu \in \cS_r(\lambda)^\mathrm{C}} \mu W^\mathrm{C} \right) \right)\\
				&\subset \left(\bigcup_{\mu \in \cS_r(\lambda)}\mu \partial \mathring{W}\right) \cup \left(\bigcup_{\mu \in \cS_r(\lambda)^\mathrm{C}} \mu \partial W^\mathrm{C} \right) = \bigcup_{\mu \in S_r} \mu \partial W.
			\end{align*}
			
			Therefore every connected component of $ W \setminus \bigcup_{\mu\in\cS_r} \mu\partial W$ is contained in some $W_r(\lambda)$, so
			\begin{equation*}
				\left\vert A^H_r\right\vert = \vert W_r \vert \leq \#\pi_0 \left( W \setminus \bigcup_{\mu\in\cS_r} \mu\partial W\right). 
			\end{equation*}
		\end{proof}
	\end{Lemma} 
	
	\begin{Lemma}
		For a model set $\Lambda(G,H,\Gamma, W)$ with polytopal window $W \subset H$ we have
		\begin{equation*}
			\#\pi_0 \left( W \setminus \bigcup_{\mu\in\cS_r} \mu\partial W\right) \leq \# \pi_0 \left( H \setminus \bigcup_{\mu\in\cS_r} \bigcup_{i=1}^N \mu P_i\right).
		\end{equation*}
		\begin{proof}
			Since $\bigcup_{\mu\in\cS_r} \mu\partial W \subset  \bigcup_{\mu\in\cS_r} \bigcup_{i=1}^N \mu P_i$ we have
			\begin{equation*}
				\#\pi_0 \left( W \setminus \bigcup_{\mu\in\cS_r} \mu\partial W\right) \leq \# \pi_0 \left( W \setminus \bigcup_{\mu\in\cS_r} \bigcup_{i=1}^N \mu P_i\right).
			\end{equation*}
			Since $e \in \cS_r$ for all $r$ we have $\partial W \subset \bigcup_{\mu \in S_r}\bigcup_{i=1}^N \mu P_i$ such that all regions inside $W$ stay the same if we increase $W$ to $H$. If we add the regions outside of $W$ we therefore get
			\begin{equation*}
				\#\pi_0 \left( W \setminus \bigcup_{\mu\in\cS_r} \bigcup_{i=1}^N \mu P_i\right) \leq \# \pi_0 \left( H \setminus \bigcup_{\mu\in\cS_r} \bigcup_{i=1}^N \mu P_i\right).
			\end{equation*}
		\end{proof}
	\end{Lemma}
	
	By \cref{Lem:HomLieTop} the topology on $H$ is the same as on $\RR^{\dim(H)}$ and we additionally assumed $H$ to be non-crooked. So the problem of counting the connected components is a well known problem from the theory of real hyperplane arrangements. A general upper bound for the number of connected components has been known for a long time and first appears in \cite{Schlafli} by L. Schläfli, see also \cite[Theorem 1.2]{Dimca}. For an arrangement $\cH \subset \RR^n$ consisting of $k$ different hyperplanes we get the general upper bound
	\begin{equation*} 
		\sum_{i=0}^n \begin{pmatrix} k \\ i\end{pmatrix} \asymp k^n.
	\end{equation*}
	In our case $n= \dim(H)$ and $k= \vert \cS(r)\vert$. We know from \cref{Prop:Growth} that ${\vert \cS(r)\vert \asymp r^{\homdim(G)}}$. Combining these results yields
	\small\begin{equation*}
		p(r) = \left\vert A_r^H\right\vert \leq \#\pi_0 \left( H \setminus \bigcup_{\mu\in\cS(r)} \bigcup_{i=1}^N \mu P_i\right) \ll \vert \cS(r)\vert^{\dim(H)} \asymp r^{\homdim(G) \cdot \dim (H)}.
	\end{equation*}\normalsize
	
	\subsection{Lower bound of the growth}\label{SubSec:LowerBound}
	We fix some parameters of the window which will help us in proving \cref{MainTHM}.
	\begin{Definition}
		For a given polytopal window $W$ we fix the following parameters:
		\begin{enumerate}[label=\roman*)]
			\item A \textit{center of the window} $c_W \in W$ such that $\inf\{r \in \RR \mid B_r(c_W) \subset W \}$ is maximal,
			\item the \textit{inner radius} of the window $I_W:=\sup\{r \in \RR \mid B_r(c_W) \subset W\}$,
			\item the \textit{outer radius} of the window $O_W:=\inf\{r \in \RR\mid W \subset B_r(c_W)\}$,
			\item the \textit{size of $\partial_i W$} 
			\begin{equation*} 
				F_i:=\sup\{r \in \RR \mid \exists p \in P_i: B_r(p)\cap \partial_i W = B_r(p) \cap P_i\}
			\end{equation*}
			and the minimum of all the sizes of the faces 
			\begin{equation*} 
				F_W:=\min\{F_i \mid i \in \{1,...,N\}\},
			\end{equation*}
			\item for each face $\partial_i W$ a \textit{face center} $p_i \in \partial_i W$ such that $B_{F_W}(p_i)\cap \partial_i W=B_{F_W}(p_i)\cap P_i$. 
		\end{enumerate}
	\end{Definition}
	We will use this parameters in our proof later. The centres may not be unique but we fix a choice for the rest of the argument. Further we need to widen the definition of parallel a bit, since we are only interested in intersections inside a bounded region.
	\begin{Definition}
		Let $B \subset H$ be a bounded region and $P_1, P_2$ two hyperplanes in $H$. We call $P_1$ and $P_2$ \textit{almost parallel with respect to $B$} if $P_1\cap P_2 \cap B = \emptyset$. 
	\end{Definition}
	The aim is now to find a region inside $W$ for which it makes no difference if it is divided by a face $\partial_i W$ or by the whole hyperplane $P_i$. Further we wish to get a one to one correspondence between the connected components and the acceptance domains in this small region.
	\begin{Definition}
		Let $B \subset H$ be a bounded region. For $s \in H$ we say $s \partial_i W$ \textit{cuts $B$ fully} if
		\begin{equation*}
			(s \partial_i W) \cap B = (s  P_i) \cap B \neq \emptyset.
		\end{equation*} 
		And if additionally
		\begin{equation*}
			(s  P_i^+) \cap B = (s  W) \cap B \neq \emptyset,
		\end{equation*} 
		we say  $s \partial_i W$ \textit{cuts $B$ all-round}.  
	\end{Definition}
	
	\begin{Remark}
		An all-round cut is always a full cut, but the converse is false, see \cref{Fig:Cuts}.
	\end{Remark}
	
	\begin{figure}
		\centering
		\captionsetup{width=0.75\linewidth}
		\begin{tikzpicture}[descr/.style={fill=white,inner sep=2.5pt}]
			\draw (0,0) circle (2);
			\draw[-] (-2,-1) -- (2,1);
			\draw[-, dotted] (-3,-1.5) -- (-2,-1); 
			\draw[-, dotted] (2,1) -- (3,1.5);
			\draw[-] (-2,-1) -- (-2,-2); 
			\draw[-] (2,1) -- (3,0);
			\fill[fill=gray!50, fill opacity=0.2] (-2,-2.5) -- (-2,-1) -- (2,1) -- (3,0) -- (3,-2.5) -- (-2,-2.5);
			\node at (-1.2,1.3) {$B$};
			\node at (2.5, -2) {$W$};
			\node at (2.8, 1.7) {$P_i$};
		\end{tikzpicture}
		\hspace{0.5cm}
		\begin{tikzpicture}[descr/.style={fill=white,inner sep=2.5pt}]
			\draw (0,0) circle (2);
			\draw[-] (-2,-1) -- (2,1);
			\draw[-, dotted] (-3,-1.5) -- (-2,-1); 
			\draw[-, dotted] (2,1) -- (3,1.5);
			\draw[-] (-2,-1) -- (-2.5,-2); 
			\draw[-] (2,1) -- (-1,-2);
			\fill[fill=gray!50, fill opacity=0.2] (-2.75,-2.5) -- (-2,-1) -- (2,1) -- (-1.5,-2.5) -- (-2.75,-2.5);
			\node at (-1.2,1.3) {$B$};
			\node at (-1.7, -2) {$W$};
			\node at (2.5, 1.6) {$P_i$};
		\end{tikzpicture}
		\caption{On the left $\partial_i W$ cuts $B$ fully and all-round on the right $\partial_i W$ cuts $B$ fully but not all-round.}
		\label{Fig:Cuts}
	\end{figure}
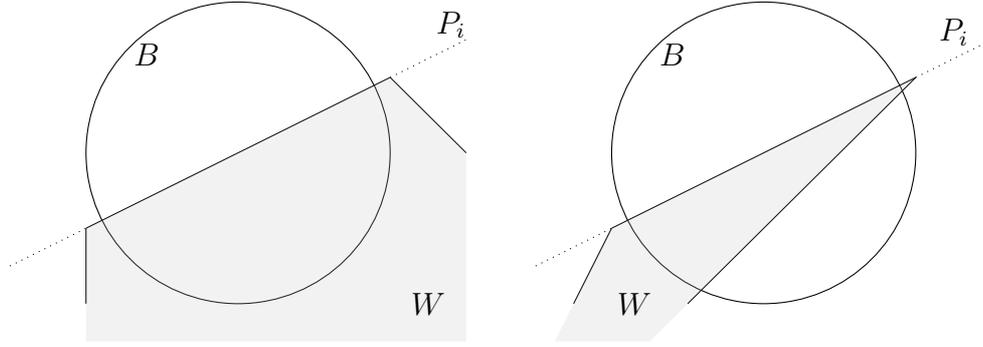
	
	Additionally to decreasing the area of interest we will decrease the set from which we operate on the window, instead of considering all elements from $\cS_r$  we define for each face a subset $U_i(r) \subset \cS_r$. The $U_i(r)$ will be defined such that we only obtain all-round cuts and get a one to one correspondence between the connected components and the acceptance domains.
	\begin{Definition}\label{Def:SmallBall}
		Let $(k,h) \in \RR^2$. The region we will consider is $B_h(c_W)$ and the set from which we operate is ${U_i:= B_k(c_W p_i^{-1})}$ for all $i\in \{1,...,N\}$. If the following conditions are fulfilled we call $(k,h)$ a \textit{good pair}:
		\begin{enumerate}[label=\roman*)]
			\item $0 < k <h$,
			\item $h < I_W$, therefore $B_h(c_W) \subset W$,
			\item $\forall a\in B_{O_W}(e)$, $x \in B_{2h}(e)$: $\vert a x a^{-1}\vert_H \leq F_W$,
			\item $\forall i \in \{1,...,N\}$: $\forall s \in U_i$: $(s  P_i^+) \cap B_h(c_W) = (s  W) \cap B_h(c_W)$.
		\end{enumerate}
	\end{Definition}
	
	\begin{Remark}
		Observe that if $(k,h)$ is a good pair, then $(k',h)$ is a good pair for all $0 < k' < k$.
	\end{Remark}
	
	\begin{Proposition}\label{Prop:GoodPair}
		A good pair exists.
	\end{Proposition}
	
	For the proof we need some preparation. It will begin after \cref{Cor:Angle}
	
	\begin{Lemma}\label{Cor:ConBall}
		Let $G$ be a homogeneous Lie group and $x \in G$ fixed. Then for all $\epsilon >0$ there exists $\delta(x) >0$ such that for $u \in B_{\delta(x)}(e)$:
		\begin{equation*}
			x u x^{-1} \in B_{\epsilon}(e).
		\end{equation*}
		Further if $\vert x \vert \leq k$ then there exists a $\delta'(k)>0$ such that for $u \in B_{\delta'(k)}(e)$:
		\begin{equation*}
			x u x^{-1} \in B_{\epsilon}(e).
		\end{equation*}
		\begin{proof}
			Using the Baker-Campbell-Hausdorff formula, we get
			\small\begin{align*}
				x u &x^{-1} = \left(x + u + \frac{1}{2}[x,u]+\frac{1}{12}([x,[x,u]]-[u,[x,u]])-...\right) x^{-1}\\
				&= \left(x + u + \frac{1}{2}[x,u]+\frac{1}{12}([x,[x,u]]-[u,[x,u]])-...\right)\\
				&\quad\quad - x + \frac{1}{2}\left[\left(x + u + \frac{1}{2}[x,u]+\frac{1}{12}([x,[x,u]]-[u[x,u]])-...\right) ,x^{-1}\right]+...\\
				&= u +B(x,u),
			\end{align*}\normalsize
			where $B(x,u)$ only contains terms which include $[u,x]$. The continuity of the Lie bracket implies the claim. Be aware that we work in exponential coordinates here, as explained in \cref{Sec:HomLie}, so formally we should write $\exp(x)$ and $\exp(u)$ instead of $x$ and $u$.
		\end{proof}
	\end{Lemma}
	
	\begin{Lemma}\label{Lem:CutsFully}
		Let $(k,h)$ fulfil conditions i) and iii) of \cref{Def:SmallBall}, then for any $i\in \{1,...,N\}$ and for every $s \in U_i$ it holds that $s P_i $ cuts $B_h(c_W)$ fully.
		\begin{proof}
			First we show that for all $s \in U_i$ we get $s P_i \cap B_h(c_W) \neq \emptyset$. We can write $s=a \cdot c_W \cdot p_i^{-1}$ with $a \in B_k(e)$. Then
			\begin{equation*}
				d(s \cdot p_i, c_W)= \left\vert a \cdot c_W \cdot p_i^{-1} \cdot p_i \cdot c_W^{-1}\right\vert = \vert a \vert < k < h.
			\end{equation*}
			Now we need to show that $s \partial_i W \cap B_h(c_W) = s P_i \cap B_h(c_W)$. This is equivalent to
			\begin{equation*}
				\partial_i W \cap s^{-1} B_h(c_W) =  P_i\cap s^{-1} B_h(c_W).
			\end{equation*}
			The inclusion $\subseteq$ is obvious since $\partial_i W \subset P_i$. We show that $s^{-1} B_h(c_W) \subseteq B_{F_W}(p_i)$, then the claim follows from the definition of $p_i$ and $F_W$. Let $x \cdot c_W \in B_h(c_W)$ be an arbitrary element and $s=a \cdot c_W \cdot p_i^{-1}$ as above. 
			\begin{equation*}
				d(s^{-1} x c_W, p_i)= \big\vert \hspace{-0.45cm}\underbrace{p_i c_W^{-1}}_{=: y \in B_{O_W}(e) } \cdot \underbrace{a^{-1} x_{}}_{\in B_{h + k}(e)} \cdot \underbrace{c_W p_i^{-1}}_{=y^{-1}}\big\vert \leq F_W.
			\end{equation*}
			The inequality follows by iii) of \cref{Def:SmallBall}.
		\end{proof}
	\end{Lemma}
	
	From the proof we can extract the following corollary.
	
	\begin{Corollary}\label{Cor:IntersectB_k}
		Let $i\in \{1,...,N\}$, for every $s \in U_i$ we have that $s P_i $ intersects $B_k(c_W)$ non-trivially.
	\end{Corollary}
	
	\begin{Corollary}
		Let $(k,h)$ fulfil conditions i), iii) and iv) of \cref{Def:SmallBall}, then for any $i\in \{1,...,N\}$ and for every $s \in U_i$ holds $s P_i $ cuts $B_h(c_W)$ all-round.
	\end{Corollary}
	
	We will also need the definition of an intersection angle between to hyperplanes, since we will show that by acting with a small element, we can only rotate a plane a bit.
	
	\begin{Definition}\label{Def:Angle}
		The angle between two hyperplanes $P$ and $Q$ in $\RR^d$ with normals $n_P$ and $n_Q$, both normalized, is given by
		\begin{equation*}
			\sphericalangle(P,Q) := \cos^{-1} \left(\vert \langle n_P \cdot n_Q \rangle \vert\right).
		\end{equation*} 
		For $i,j \in \{1,...,N\}$ with $i \neq j$ we denote by $\alpha_{ij}$ the angle between $c_Wp_i^{-1} P_i$ and $c_W p_j^{-1} P_j$, i.e. 
		\begin{equation*}
			\alpha_{ij}:=\sphericalangle(c_W p_i^{-1} P_i, c_W p_j^{-1} P_j).
		\end{equation*}
	\end{Definition}
	
	\begin{Remark}\label{Rem:FixFamily} 
		In the definition we use $c_W p_i^{-1} P_i$ instead of $P_i$ since this plane is sort of the prototype for the family $U_i P_i$, all the other planes from this family then result from an action with a small element. Since $u \in U_i$ is of the form $a c_Wp_i^{-1}$ with $a\in B_k(e)$.
	\end{Remark}	
	
	\begin{Convention}
		We choose $i_1,..., i_{\dim(H)}$ such that 
		\begin{equation*}
			\bigcap_{l=1}^{\dim(H)} c_W p_{i_l}^{-1} P_{i_l}= \{c_W\}.
		\end{equation*}
		So this is a set of hyperplanes in which each intersection of $k$ hyperplanes has dimension $\dim(H)-k$. From now on we fix such a family and denote it by $\cF$. Without loss of generality $\cF = \{P_1,...,P_{\dim(H)}\}$. 
	\end{Convention}
	
	\begin{Lemma}\label{Lem:Angle}
		For all $r>0$ there exists $\beta(r)$, with $\beta(r) \to 0$ for $r \to 0$, such that for all $x \in B_r(e) \subset H$ and any hyperplane $P$ we have $\sphericalangle(x P, P) \leq \beta(r)$.
		\begin{proof}
			Since $H$ is a non-crooked homogeneous Lie group we know that $xP$ is again a hyperplane. So let 
			\begin{equation*}
				P= \left\{a + \sum_{i=1}^n t_i v_i \mid t_i \in \RR \right\}
			\end{equation*}
			where $a, v_i \in \RR^n$. By the form of the group action, which we discuss in \hyperref[Appendix:Poly2]{Appendix B.1}, we know that $xP$ is of the form
			\begin{equation*}
				xP = \big(f_1(x,P),...,f_n(x,P)\big)^\mathrm{T}
			\end{equation*}
			with $f_i$ polynomials of a special form, namely
			\begin{equation*}
				f_i(x,P) = x_i + P_i +\sum_{k=1}^n\sum_{\substack{\alpha_1,...,\alpha_n \in \NN \\ \sum \alpha_i \neq 0}} c_{k,\alpha_1,...,\alpha_n} P_k x_1^{\alpha_1}...x_n^{\alpha_n}.
			\end{equation*}
			The direction vectors are the ones from $P$ plus some deviation which depends on $x$. If $x$ gets smaller the two planes are getting closer to being parallel. 
		\end{proof}
	\end{Lemma}
	
	\begin{Corollary}\label{Cor:Angle}
		For all $r>0$ there exists $\beta(r)$, with $\beta(r) \to 0$ for $r \to 0$, such that for all $x, y \in B_r(e) \subset H$ and any hyperplane $P$ we have $\sphericalangle(x P, y P) \leq 2\beta(r)$.
	\end{Corollary}

	\begin{proof}[Proof of \cref{Prop:GoodPair}]
		By \cref{Cor:ConBall} there exists an upper bound $b_1$ on $h$ such that for all $a\in B_{O_W}(e)$, $x \in B_{2h}(e)$: $\vert a x a^{-1}\vert_H \leq F_W$. Set $h':=\min\{b_1, \frac{I_W}{2}\}$ and set $k':=\frac{h'}{2}$ then the conditions i), ii) and iii) from \cref{Def:SmallBall} are fulfilled.\par 
		By \cref{Lem:CutsFully} for any $i \in \{1,...,N\}$ and all $s \in B_{k'}(c_Wp_i^{-1})$ we have $sP_i$ cuts $B_{h'}(c_W)$ fully, this also holds for all $h \leq h'$.\par 
		Now assume that there is a cut which is full but not all-round, therefore 
		\begin{equation*}
			s P_i^+ \cap B_{h'}(c_W) \neq sW \cap B_{h'}(c_W).
		\end{equation*}
		To be more precise we have $\ sW \cap B_{h'}(c_W) \subsetneq s P_i^+ \cap B_{h'}(c_W)$ since $sW \subset sP_i^+$. Let 
		\begin{equation*} 
			b_{s}^i:= \inf\{r \in \RR \mid  \exists x \in B_r(c_W) : x \in s P_i^+, x \notin sW\}.
		\end{equation*}
		We see that $b_s^i \neq 0$ since $sW$ is a polytope with non-empty interior. Now set 
		\begin{equation*} 
			h:=\min\big\{h', \inf_{s\in B_{k'}(c_Wp_i^{-1})}\{b_s^i\}\big\}
		\end{equation*}
		and $k=\frac{h}{2}$. The last thing to observe now is that the infimum over the $b_s^i$ is not zero. If it were zero this would mean that the polytope $sW$ could become arbitrary thin such that only an even smaller ball would fit in. But this can not be the case since we have seen that we can only rotate the bounding hyperplanes only by a small amount.
	\end{proof}
	
	\begin{Convention}
		From now on let $(k,h)$ be a good pair.
	\end{Convention}
	
	Observe that $U_i \subset WW^{-1}$ for all $i\in \{1,...,N\}$. Further notice that we operate differently on the different hyperplanes which bound $W$, the $U_i$ may overlap but they are not equal. Additionally we have chosen $B_h(c_W)$ so that for each of the hyperplanes it does not make a difference if we operate on the face $\partial_i W$ or on the hyperplane $P_i$. \par \medskip
	
	Now we will reconsider the dependence of the growing parameter $r$ and the lattice $\Gamma$.
	
	\begin{Definition}
		Set $U_i(r):= \pi_H((B_r^G(e) \times U_i)\cap \Gamma)$ which is a finite subset of $U_i$.
	\end{Definition}
	
	\begin{Remark}
		Observe that $U_i(r)$ is a subset of the $r$-slab $\cS_r$, since $U_i \subset WW^{-1}$.
	\end{Remark}
	
	\begin{Proposition}\label{Lem:CutsAllround}
		The number of connected components of $ B_h(c_W) \setminus \bigcup_{i=1}^N \bigcup_{s \in U_i(r)} s \partial_i W$
		is a lower bound of the number of acceptance domains $\vert A_r^H \vert$, i.e.
		\begin{equation*}
			\# \pi_0 \left( B_h(c_W) \setminus \bigcup_{i=1}^N \bigcup_{s \in U_i(r)} s \partial_i W \right) \leq \vert A_r^H \vert.
		\end{equation*}
		\begin{proof}
			Recall that a pre-acceptance domain $A_r^H(\lambda)$ is contained in an acceptance domain $W_r(\lambda)$. Let $C$ be a connected component of $B_h(c_W) \setminus \bigcup_{i=1}^N \bigcup_{s \in U_i(r)} s \partial_i W$. By \cref{Lem:CutsFully} we can replace the faces by the hyperplanes without changing the connected components in $B_h(c_W)$, so we consider ${B_h(c_W) \setminus \bigcup_{i=1}^N \bigcup_{s \in U_i(r)} s P_i}$.\par 
			We show that if an acceptance domain intersects a connected component of $B_h(c_W) \setminus \bigcup_{i=1}^N \bigcup_{s \in U_i(r)} s P_i$ it is fully contained in it. Let $C'$ be another connected component of $B_h(c_W) \setminus \bigcup_{i=1}^N \bigcup_{s \in U_i(r)} s P_i$ and assume that $C \cap W_r(\lambda) \neq \emptyset \neq C' \cap W_r(\lambda)$. Between $C$ and $C'$ is a hyperplane $s P_i$ for some $i\in \{1,...,N\}$ and $s \in U_i(r)$. Therefore $C \subset s \mathring{P^+}$ and $C'\subset s \mathring{P^-}$ or the other way around. And since the cut $s\partial_i W$ is all-round we get that $C \subset s \mathring{W}$ and $C' \subset s W^\mathrm{C}$ or the other way around. But either $W_r(\lambda) \subset s\mathring{W}$ or $W_r(\lambda) \subset sW^\mathrm{C}$, a contradiction.	
		\end{proof}
	\end{Proposition}
	
	It remains to find a combinatorial argument for counting the connected components of 
	\begin{equation*}
		B_h(c_W) \setminus \bigcup_{i=1}^N \bigcup_{s \in U_i} s \partial_i W=B_h(c_W) \setminus \bigcup_{i=1}^N \bigcup_{s \in U_i} s P_i,
	\end{equation*}
	which yields a lower bound by \cref{Lem:CutsAllround}. This will be done in the next section.\par \medskip
	
	In the rest of the section we will prove the following proposition, which gives us the tools for the combinatorics in the next section.
	
	\begin{Proposition}\label{Prop:ToolComb}
		There exists a good pair $(k,h)$ such that:
		\begin{enumerate}
			\item For all $I \subset \{1,....,N\}$ with $\vert I \vert = \dim(H)$ and all $u_{1} \in U_{i_1},...,u_{{\dim(H)}} \in U_{i_{\dim(H)}}$ we get
			\begin{equation*}
				u_{1} P_{i_1} \cap ... \cap u_{{\dim(H)}} P_{i_{\dim(H)}} = \{s\}, \text{where } s \in B_h(c_W).
			\end{equation*}
			\item For every constant $c>0$ and all $s \in H$, there is a $r_0$ such that for all $r > r_0$ we get that
			\begin{equation*}
				\big\vert \{u \in U_i(r) \mid s \in u P_i\}\big\vert \leq c \vert U_i(r) \vert.
			\end{equation*}
		\end{enumerate}
	\end{Proposition}
	
	\begin{Lemma}\label{Lem:IntersectionPoint}
		For $i,j \in \{1,...,\dim(H)\}$, $i \neq j$, there exists a good pair $(k,h)$ such that for all $u \in U_i=B_k(c_Wp_i^{-1}), v \in U_j$ we have $u P_i$ and $vP_j$ are not almost parallel with respect to $B_h(c_W)$.  
		\begin{proof}
			Fix some $i,j \in \{1,...,N\}$ with $i \neq j$. By \cref{Cor:IntersectB_k} all the $u P_i$, $v P_j$ with $u \in U_i$, $v \in U_j$  intersect $B_k(c_W)$.\par 
			Further we can control the angle between the two hyperplanes by \cref{Lem:Angle} so that for all $ u \in U_i, v \in U_j $:
			\begin{equation*}
				\sphericalangle(u P_i, v P_j) \geq \sphericalangle(P_i, P_j) - \sphericalangle(u P_i, P_j) - \sphericalangle(P_i, v P_j)\geq  \alpha_{ij} - 2 \beta(k),
			\end{equation*}
			where $\beta(k)$ is from \cref{Lem:Angle}. We can choose $k$ so small such that $0 < \alpha_{ij} - 2 \beta(k) < \frac{\pi}{2}$, this means that the hyperplanes can not be parallel so they intersect somewhere. For two hyperplanes which intersect the same ball of radius $k$ and which intersect in at least a given angle there is a bound for the intersection to the center point of the ball
			\begin{equation*}
				c(k):= k \left( 1+\frac{1}{\tan\left(\frac{\alpha_{ij}-2 \beta(k)}{2}\right)}\right).
			\end{equation*}
			The idea how to establish this bound is to consider the space which is orthogonal to the intersection of $u P_i$ and $v P_j$ and contains $c_W$. Then one can argue in a two-dimensional plane.\par
			The bound $c(k)$ goes to zero if $k$ goes to zero, so we can choose $k$ so small that $c(k) < h$. Therefore the two planes intersect inside $B_h(c_W)$.
		\end{proof}
	\end{Lemma}
	
	\begin{Corollary}\label{Cor:IntersectionPoint}
		There exists a good pair $(k,h)$ such that for all $u_i \in U_{i}$ and $i \in \{1,...,\dim(H)\}$ we can find some $x \in B_h(c_W)$ such that:
		\begin{equation*}
			\bigcap_{i=1}^{\dim(H)} u_i P_i = \{x\}.
		\end{equation*}
		\begin{proof}
			By the choice of the family $\cF$ we know that $\bigcap_{i=1}^{\dim(H)} c_W p_{i}^{-1} P_{i}= \{c_W\}$, we will first show that there is a $k_0$ such that for all $0< k \leq k_0$ we also get a zero dimensional intersection inside $B_h(c_W)$ if we replace $c_W p_i^{-1}$ by $u_i \in U_i=B_k(c_Wp_i^{-1})$.\par 
			This intersection behaviour means that if we choose some vector $v \parallel c_Wp_i^{-1}P_i$, then $v \parallel c_Wp_j^{-1}P_j$ can maximally hold for all but one $j$. Since otherwise the intersection of all hyperplanes would end up in a line instead of a point. We have to choose $k_0$ such that for all $i\in \{1,...,\dim(H)\}$ and all $v \parallel u_i P_i$, $u_i \in U_i$, there exists a $j\in \{1,...,\dim(H)\}\setminus\{i\}$ and a $u_j \in U_j$ such that $v \nparallel u_jP_j$. Since operating with an element from $U_i$ only rotates the hyperplane a little it is possible to find such a $k_0$ and than the property also holds for all $k$ smaller than $k_0$.\par 
			Now we have to check that the intersection point also lies inside of $B_h(c_W)$. We do this stepwise. It is clear that $\bigcap_{i=1}^{\dim(H)} c_W p_{i}^{-1} P_{i}= \{c_W\}$ and $c_W \in B_h(c_W)$. Now we change $c_W p_1^{-1}$ to some $u_1 \in U_1$ and consider $u_1 P_1 \cap \bigcap_{i=2}^{\dim(H)} c_W p_{i}^{-1} P_{i} =\{x_1\}$. We already know that $\bigcap_{i=2}^{\dim(H)} c_W p_{i}^{-1} P_{i}$ is a subspace of dimension $1$ and that $u_1 P_1$ intersects this subspace. Since $u_1= a_1 c_W p_1^{-1}$ with $a_1\in B_k(e)$ the plane $u_1 P_1$ is just a small shift, this follows from the form of the group action, which we explained in \cref{Sec:HomLie}, and a small rotation away from $c_W p_1^{-1} P_1$, this follows from \cref{Lem:Angle}. Therefore $d(x_1, c_W) < \epsilon_1(k)$, where $\epsilon_1$ depends on $k$ and goes to zero if $k$ goes to zero. We can iterate this process and get a new solution on each step until we end at $x_d$, $d=\dim(H)$, where we have
			\small\begin{align*}
				d(x_d,c_W) &< d(x_d, x_{d-1})+ d(x_{d-1},x_{d-2})+... + d(x_2,x_1)+d(x_1,c_W) < \sum_{i=1}^d \epsilon_i(k) =:\epsilon(k).
			\end{align*}\normalsize
			So by choosing $k$ such that $\epsilon(k)<h$ we get the claim. This is possible since $\epsilon(k) \to 0$ for $k \to 0$.
		\end{proof}
	\end{Corollary}
	
	The corollary tells us that all intersections result in a single point in $B_h(c_W)$, but it is not clear that different choices of $u_j$ result in different intersection points. This is a major difference to the Euclidean case, since here the action is just translation, so by acting on a hyperplane we get a parallel hyperplane, which then either is still the same hyperplane or does not intersect the original hyperplane at all.\par \medskip 
	
	To proof part (b) of \cref{Prop:ToolComb} we need \cref{Thm:Nevo}, we can use it by the discussion in \cref{SubSec:Ergodic}.
	
	\begin{Lemma}\label{Lem:BoundInsidenz}
		Consider a family $U_i(r) \cdot P_i$. For any constant $c >0$  and all $s \in H$ there is a $r_0$ such that for all $r \geq r_0$ we get that
		\begin{equation*}
			\big\vert\{u \in U_i(r) \mid s \in u P_i  \}\big\vert \leq c \cdot \vert U_i(r) \vert.
		\end{equation*}
		\begin{proof}
			Let $u \in U_i(r)$ such that $s\in u P_i$. This implies that $u^{-1} \in P_i s^{-1}$, so $u \in U_i(r) \cap (P_i s^{-1})^{-1}$. So the question is how many elements are in $U_i(r) \cap (P_i s^{-1})^{-1}$ compared to the number of elements in $U_i(r)$. To get an estimate via the Haar-measure we have to thicken $(P_i s^{-1})^{-1}$ since it is a subset of lower dimension. We consider an $\epsilon$-strip around the set, so we choose a finite set $A(\epsilon) \subset (P_is^{-1})^{-1}\cap U_i$ such that
			\begin{equation*}
				U_i \cap (P_i s^{-1})^{-1} \subset \bigcup_{p \in A(\epsilon)}B_\epsilon(p)
			\end{equation*}
			and further let $S_\epsilon:= U_i(r) \cap \cup_{p \in A(\epsilon)} B_\epsilon(p)$. We have seen that we can use \cref{Thm:Nevo} so for every $\delta >0$ and $r$ large enough
			\begin{align*}
				\delta &\geq \big\vert \vert U_i(r)\vert - \mu_{G \times H}(B_r(e) \times U_i) \big\vert = \big\vert \vert U_i(r)\vert - \mu_{G}(B_r(e)) \mu_{H}(U_i) \big\vert\\
				&= \big\vert \vert U_i(r)\vert - r^{\homdim(G)}\mu_{G}(B_1(e)) \mu_{H}(U_i)\big\vert.
			\end{align*}
			Since $S_\epsilon$ is a finite union of balls we can use the same argument for all the balls simultaneously and get for $\delta >0$ that
			\small\begin{align*}
				\lim\limits_{r \to \infty} \frac{\vert U_i(r) \cap (P_i s^{-1})^{-1}\vert }{\vert U_i(r)\vert} &< \lim\limits_{r \to \infty}\frac{\vert S_\epsilon \vert}{\vert U_i(r)\vert}
				= \frac{\sum_{p \in A(\epsilon)} r^{\homdim(G)}\mu_{G}(B_1(e))\mu_{H}(B_\epsilon(p))}{r^{\homdim(G)}\mu_{G}(B_1(e))\mu_{H}(U_i)}\\
				&= \frac{\vert A(\epsilon) \vert \cdot \mu_{H}(B_\epsilon(e))}{\mu_{H}(U_i)} \xrightarrow{\epsilon \to 0} 0.
			\end{align*}\normalsize
			
		\end{proof}
	\end{Lemma}
	
	\section{Combinatorics}\label{Sec:Combinatorics}
	The aim of this section is to give a lower bound for the number of connected components of
	\begin{equation*}
		B_h(c_W) \setminus \bigcup_{i=1}^N \bigcup_{s \in U_i(r)} s P_i,
	\end{equation*}
	under some conditions on the family $s P_i$. This will fill the gap we left in the last section. To do so we will give a short introduction to the theory of hyperplane arrangements and fix the common notation for this setup. To do so we will follow the lines of Dimca, \cite{Dimca}, and Stanley, \cite{Stanley}. Furthermore we will consider Beck's theorem which was first proved in \cite{Beck}, but also follows from the Sz\'emeredi-Trotter theorem, \cite{SzemerediTrotter}. An easier proof for the Sz\'emeredi-Trotter theorem can be found in the paper of Szekely, \cite{Szekely}, we leave it to the reader to see that the two versions are equivalent.\par\medskip
	
	\begin{Theorem}[Higher dimensional local dual of Beck's Theorem]\label{Thm:BeckInduction}
		There exists a constant $c_d$ such that for a hyperplane arrangement $\cH$ in $\RR^d$ and $B \subset \RR^d$ convex. Where $\cH$ consist of $d$ families $F_1,...,F_d$ with $\vert F_i \vert = \frac{n}{d}$ and for all $(f_1,...,f_d) \in F_1 \times ... \times F_d$ we have $B \cap \bigcap_{i=1}^d f_i = \{p\}$ for some point $p \in B$. And additionally at most $c \cdot \vert F_i \vert$ hyperplanes from $F_i$ can intersect in one point, with $0 < c < \frac{1}{100}$. Then the number of intersection points in $B$ exceeds $c_d \cdot n^d$, i.e. $\vert F_{0,B} \vert \geq c_d \cdot n^d$.
	\end{Theorem}
	
	\begin{Corollary}
		In the situation of \cref{MainTHM} and \cref{Prop:ToolComb}:
		\begin{equation*}
			\# \pi_0 \left( B_h(c_W) \setminus \bigcup_{i=1}^N \bigcup_{s \in U_i(r)} s P_i \right) \gg r^{\homdim(G) \cdot \dim(H)}.
		\end{equation*}
		\begin{proof}
			Notice that $B_h(c_W)$ is convex and the family we constructed in \cref{SubSec:LowerBound} fulfils the requirements in \cref{Thm:BeckInduction} by \cref{Prop:ToolComb}. Then by \cref{Prop:StanleyBound} the claim follows. 
		\end{proof}
	\end{Corollary}
	
	This finishes the proof of the main theorem, \cref{MainTHM}. The rest of the section is devoted to proof \cref{Thm:BeckInduction}.
	
	\begin{Definition}
		A finite set of affine hyperplanes $\cH=\{P_1,...,P_n\}$ in $\RR^d$ is called a \textit{hyperplane arrangement}.
	\end{Definition}
	
	\begin{Definition}
		Let $\cH$ be a hyperplane arrangement in $\RR^d$.
		\begin{enumerate}[label=\roman*)]
			\item A non-empty intersection of hyperplanes from $\cH$ is called a \textit{flat} of $\cH$. The set of all flats is denoted by $F(\cH)$. If we are only interested in flats of a certain dimension $k$ we denote this set by $F_k(\cH)$. Notice that the whole space is also a flat, as a result of the intersection over the empty set.
			\item The connected components of 
			\begin{equation*}
				\RR^d \setminus \bigcup_{H \in \cH} H
			\end{equation*}
			are called \textit{regions of the arrangement}. The set of all regions is denoted by $f_d(\cH)$ and the number of these regions is denoted by $r(\cH)$.
			\item Let $B\subset \RR^d$ then the connected components of 
			\begin{equation*}
				B \setminus \bigcup_{H \in \cH} H
			\end{equation*}
			are called \textit{regions of the arrangement with respect to $B$} and the number of these regions is denoted by $r_B(\cH)$.
			\item Let $B\subset \RR^d$, then define the \textit{arrangement with respect to $B$} as 
			\begin{equation*}
				\cH_B := \{H \in \cH \mid H \cap B \neq \emptyset\}.
			\end{equation*}
			\item Let $B\subset \RR^d$ then a \textit{flat with respect to $B$} is a flat of $\cH$ which intersects $B$. The set of these flats is denoted by $F_B(\cH)$ and again if we only consider the flats of dimension $k$, we write $F_{k,B}(\cH)$.
		\end{enumerate} 
	\end{Definition}
	
	\begin{Remark}
		Obviously $r_B(\cH)$ only depends on the hyperplanes which intersect $B$, so $r_B(\cH)=r_B(\cH_B)$. Further notice that in general $r_B(\cH) \neq \vert \{ R \in f_d(\cH) \mid R \cap B \neq \emptyset\}\vert$, but the following proposition will show that equality holds for convex $B$. 
	\end{Remark}
	
	\begin{Proposition}\label{Prop:RegionsWithRespectB}
		Let $\cH$ be a hyperplane arrangement in $\RR^d$ and $B \subset \RR^d$ convex, then
		\begin{equation*}
			r_B(\cH) = \vert\{R \in f_d(\cH) \mid R \cap B \neq \emptyset\}\vert.
		\end{equation*}
		\begin{proof}
			The regions of a hyperplane arrangement are convex, since $B$ is also convex we have for all $R \in f_d(\cH)$ that $R \cap B$ is convex and especially connected. So each region of the arrangement either intersects $B$ and therefore contributes exactly one to $r_B(\cH)$ or it does not intersect at all.
		\end{proof}
	\end{Proposition}
	
	\begin{Definition}
		An arrangement $\cH$ in $\RR^d$ is called:
		\begin{enumerate}[label=\roman*)]
			\item \textit{Central} if $\bigcap_{H \in \cH} H \neq \emptyset$ \footnote{observe that the empty arrangement is central since the empty intersection is the whole space},
			\item \textit{central with respect to $B$}, for some $B \subset  \RR^d$, if $B \cap \bigcap_{H \in \cH} H \neq \emptyset$. 
		\end{enumerate}
	\end{Definition}
	
	We will now define the characteristic polynomial for the arrangement $\cH$ which depends on $B\subset \RR^d$, this idea is similar to the standard idea of considering the characteristic polynomial of the arrangement.
	
	\begin{Definition}
		Let $\cH$ be a hyperplane arrangement in $\RR^d$. The \textit{characteristic polynomial with respect to $B$} is defined by
		\begin{equation*}
			\chi_{\cH,B}(t):= \hspace{-0.5cm }\sum_{\substack{\cA \subset \cH\\ \cA \text{ central with respect to } B}} \hspace{-0.5cm } (-1)^{\vert \cA \vert} t^{\dim(\cap_{H \in \cA} H)}.
		\end{equation*}
	\end{Definition}
	
	Following the argumentation in \cite{Dimca} for the characteristic polynomial with respect to $B$ instead of the characteristic polynomial we can establish the following theorem. Since the argument is exactly the same we will not prove the statement here.
	
	\begin{Theorem}[{\cite[Theorem 2.8]{Dimca}}]\label{Thm:RegionsCharPoly}
		Let $\cH$ be a hyperplane arrangement in $\RR^d$ and $B \subset \RR^d$ convex, then 
		\begin{equation*}
			r_{B}(\cH) = (-1)^d \chi_{\cH,B}(-1).
		\end{equation*}
	\end{Theorem} 
	
	We can use this formula with the help of the following lemma, which we import from Stanley.
	
	\begin{Lemma}[{\cite[Theorem 3.10.]{Stanley}}]\label{Lem:Stanley}
		Let $\chi(t)$ be a characteristic polynomial of an hyperplane arrangement $\cH$ in $\RR^d$, then
		\begin{equation*}
			\chi(t) = \sum_{f \in F_B(\cH)} a_f t^{\dim(f)}
		\end{equation*}
		with $(-1)^{d-\dim(f)} a_f >0$ for $f \in F_B(\cH)$.
	\end{Lemma}
	
	\begin{Corollary}\label{Prop:StanleyBound}
		Let $\cH$ be a hyperplane arrangement in $\RR^d$ and $B \subset \RR^d$ convex, then
		\begin{equation*}
			r_B(\cH) \geq \vert F_B(\cH) \vert \geq \vert F_{0,B}(\cH) \vert.
		\end{equation*}
		\begin{proof}
			By \cref{Thm:RegionsCharPoly} and \cref{Lem:Stanley} it is 
			\begin{equation*}
				r_{B}(\cH) = (-1)^d \chi_{\cH,B}(-1) = \sum_{f \in F_B(\cH)} (-1)^d a_f (-1)^{\dim(f)} > \sum_{f \in F_B(\cH)} 1 = \vert F_B(\cH) \vert.
			\end{equation*}
		\end{proof}
	\end{Corollary}
	
	This proposition means that to establish a lower bound of the regions it is enough to count the number of intersection points. We will do this by following the idea of the proof of Beck's theorem, \cite{Beck}. But instead of considering the set of lines/hyperplanes spanned by a point set we turn all the arguments around and consider the intersection points of a given arrangement. We will first handle the case of dimension two and then use induction to generalize the statement.\par\medskip
	We first state the Szemer\'edi-Trotter Theorem in two equivalent ways.
	
	\begin{Theorem}[Szemer\'edi-Trotter Theorem, {\cite{Szekely}}, {\cite{SzemerediTrotter}}]\label{Thm:SzemerediTrotter}~\par
		Let $n,m \in \NN$, we set
		\begin{equation*}
			I(n,m)= \max_{\vert P \vert =n, \vert L \vert =n} \vert\{(p,l) \in P \times L \mid p \in l\} \vert,
		\end{equation*}
		where $P$ denotes a set of points and $L$ a set of lines in $\RR^2$.
		\begin{enumerate}[label=\roman*)]
			\item There exists a constant $c>0$ such that $I(n,m) < c \cdot (n^{2/3}m^{2/3} + n + m)$.
			\item Let $\sqrt{n} \leq m \leq \begin{pmatrix} n \\ 2 \end{pmatrix}$, then there exists a constant $c>0$ such that \linebreak ${I(n,m) < c \cdot (n^{2/3}m^{2/3})}$.
		\end{enumerate}
	\end{Theorem}
	
	\begin{Remark}
		The constant for the growth in the Szemer\'edi-Trotter Theorem is known to be less than $2.5$ but more than $0.4$.
	\end{Remark}
	
	\begin{Definition}
		Let $\cH$ be a hyperplane arrangement then for a flat $f \in F(\cH)$ we define $S(f):=\{H \in \cH \mid f \subset H \}$, and further $a(f):=\vert S(f) \vert$.
	\end{Definition}
	
	\begin{Definition}
		For a hyperplane arrangement $\cH$ let
		\begin{align*}
			t(\cH,k) &:= \vert \{p \in F_0(\cH) \mid a(p) \geq k\}\vert, \\
			t^\ast(\cH,k) &:= \vert \{p \in F_0(\cH) \mid k \leq a(p) < 2k\}\vert. 
		\end{align*}
		Further we consider the maximal value of the two terms
		\begin{align*}
			t(n,k) &:= \max_{\vert \cH \vert = n } t(\cH,k), \\
			t^\ast(n,k) &:= \max_{\vert \cH \vert = n } t^\ast(\cH,k) . 
		\end{align*}
	\end{Definition}
	
	For our proof we need some bounds on these terms. The first two are from the paper of Beck, \cite{Beck}, and are easy to prove. The third is a corollary of the Szemer\'edi-Trotter Theorem. Beck proves a similar inequality, \cite[Lemma 2.2]{Beck}, which is the main part of his argument. For us it would also be possible in our setup to translate the proof of Beck, but this would require a lot of effort, if the reader is interested we challenge him to follow the proof, it certainly is illuminating to translate all the arguments.
	
	\begin{Lemma}[{\cite[Lemma 2.1]{Beck}}]\label{Lem:Beck2.1}
		For a hyperplane arrangement $\cH$ in $\RR^2$, with $\vert \cH \vert = n$, we have
		\begin{align*}
			t(n,k) \leq \frac{n(n-1)}{k(k-1)},\quad &\forall 2 \leq k \leq n,\\
			t(n,k) < \frac{2n}{k},\quad &\forall \sqrt{2n} < k \leq n.
		\end{align*} 
		\begin{proof}
			For the first formula we consider the number of pairs of lines. On the one hand we consider all possible pairs and on the other hand the pairs through points in which at least $k$ lines intersect:
			\begin{equation*}
				t(n,k) \cdot \begin{pmatrix} k \\ 2 \end{pmatrix} \leq \begin{pmatrix} n \\ 2 \end{pmatrix}.
			\end{equation*}
			For the second inequality the points in which at least $k$ lines intersect are denoted by $p_1,...,p_t$. We assume that $t=\frac{2n +l }{k} \in \NN$, where $l \in \{0,...,k-1\}$. Then $t < \sqrt{2n} + \frac{l}{k}$, since $\sqrt{2n} < k$. Notice that $\vert S(p_i) \vert \geq k$ and $\vert S(p_i) \cap S(p_j) \vert \leq 1$ for $i \neq j$, since two points are connected by exactly one line. Then
			\begin{align*}
				n &= \vert \cH \vert \geq \left\vert \bigcup_{i=1}^t S(p_i) \right\vert \geq \sum_{i=1}^t \vert S(p_i)\vert - \sum_{1 \leq i < j \leq t} \vert S(p_i) \cap S(p_j)\vert\\
				&\geq \sum_{i=1}^t k - \sum_{1 \leq i < j \leq t} 1 = tk - \frac{1}{2}t (t-1) > 2n+l - \frac{1}{2} \left(\sqrt{2n}+\frac{l}{k}\right) \left(\sqrt{2n}+\underbrace{\frac{l}{k}-1}_{<0}\right)\\
				&> 2n+l - n - \sqrt{\frac{n}{2}} \frac{l}{k} = n+ l \left( 1 - \frac{\sqrt{n}}{\sqrt{2}k}\right) > n+l\left(1- \frac{1}{2}\right) \geq n.
			\end{align*}
			This is a contradiction, so $t=\lceil \frac{2n}{k}\rceil$ can not hold and also $t > \frac{2n}{k}$ is not possible since we can simply ignore some points to get the same contradiction.
		\end{proof}
	\end{Lemma}
	
	\begin{Corollary}[{Corollary of {\cref{Thm:SzemerediTrotter}}, {\cite[Theorem 2]{SzemerediTrotter}}}]\label{Cor:SzemTrot}
		For a hyperplane arrangement $\cH$ in $\RR^2$ there is some constant $\beta >0$ such that
		\begin{align*}
			t(n,k) < \beta \frac{n^2}{k^3},\quad  \forall 3 \leq k \leq \sqrt{n} .
		\end{align*}
		\begin{proof}
			Assume there are $t=\frac{c^3n^2+l}{k^3} \in \NN$, where $l\in \{0,..., k^3-1\}$ and $c=2.5$, points with $a(p)\geq k$.\par 
			Then 
			\begin{equation*}
				\sqrt{t} = n \sqrt{\frac{c^3}{k^3} + \frac{l}{n^2 k^3}} \leq n \sqrt{\frac{c^3}{k^3} + \frac{k^3-1}{n^2k^3}} < n \sqrt{\frac{2.5^3}{3^3}+\frac{1}{n^2}} < n,
			\end{equation*}
			since $n > 3$. Further
			\begin{align*}
				\begin{pmatrix} t \\ 2 \end{pmatrix} &= \frac{1}{2} \frac{c^3 n^2+l}{k^3}\left(\frac{c^3 n^2+l}{k^3}-1 \right) \geq \frac{1}{2} \left(c^3 \sqrt{n}+ \frac{l}{n^{\frac{3}{2}}}\right) \left(c^3 \sqrt{n}+ \frac{l}{n^{\frac{3}{2}}}-1 \right)\\
				&\geq \frac{1}{2} c^3 \sqrt{n}(c^3 \sqrt{n}-1) >n,
			\end{align*}
			since $c=2.5$ and $n >3$. Therefore we can use version ii) of the Szemer\'edi-Trotter Theorem: There is a constant $c$ such that $I(n,m) < c n^{2/3}m^{2/3}$ and we know that $c=2.5$ works. The $t$ point induce $t \cdot k$ incidences, but
			\begin{equation*}
				t \cdot k = c^3 \frac{n^2 + l}{k^2} \not< c n^{2/3}t^{2/3},
			\end{equation*}
			a contradiction to the theorem. Therefore $t=\lceil\frac{c^3n^2}{k^3} \rceil$ is not possible and also $t> \frac{c^3n^2}{k^3}$ is not possible by the same argument if we ignore some points. We see that $\beta =2.5^3$ is a possible choice.
		\end{proof}
	\end{Corollary}
	
	\begin{Theorem}[Local dual of Becks Theorem]\label{Thm:DualBeck}
		There exists a constant $c_2$ such that for a hyperplane arrangement $\cH$ in $\RR^2$ and $B \subset \RR^2$ convex. Where $\cH$ consist of two disjoint families $F_1$ and $F_2$ of hyperplanes with $\vert F_1 \vert = \vert F_2 \vert = \frac{n}{2}$ and such that for all $(f,g)\in F_1 \times F_2$ we have $B \cap f \cap g =\{p\}$, for some point $p \in B$ depending on $f$ and $g$. Then one of the following two cases holds:
		\begin{enumerate}
			\item There is a point $p \in B$ such that $a(p) \geq \frac{n}{100}$.
			\item The number of intersection points in $B$ exceeds $c_2 \cdot n^2$, i.e. $\vert F_{0,B}(\cH) \vert \geq c\cdot n^2$. 
		\end{enumerate}
		\begin{proof}
			We count the number of pairs of lines, we get
			\begin{equation*}
				\begin{pmatrix} n \\ 2 \end{pmatrix} \geq \sum_{p \in F_{0,B}(\cH)} \begin{pmatrix} a(p) \\ 2 \end{pmatrix} \geq \vert F_1 \vert \cdot \vert F_2 \vert = \frac{1}{4}n^2.
			\end{equation*}
			On the left side we counted all the possible options, in the middle we counted the pairs which intersect inside $B$ and on the right we counted the pairs of lines from the two families, since we know that they intersect in $B$. We will split the sum into three parts:
			\begin{align*}
				S_1 &:= \sum_{\substack{p \in F_{0,B}(\cH)\\ 2^k \leq a(p) < \sqrt{n}}} \begin{pmatrix} a(p) \\ 2\end{pmatrix},\\
				S_2 &:= \sum_{\substack{p \in F_{0,B}(\cH)\\ \sqrt{n} \leq a(p) < \frac{n}{100}}} \begin{pmatrix} a(p) \\ 2\end{pmatrix},\\
				S_3 &:= \sum_{\substack{p \in F_{0,B}(\cH)\\ 2 \leq a(p) < 2^k}} \begin{pmatrix} a(p) \\ 2\end{pmatrix} + \sum_{\substack{p \in F_{0,B}(\cH)\\ \frac{n}{100} \leq a(p) \leq n}} \begin{pmatrix} a(p) \\ 2\end{pmatrix},
			\end{align*}
			where $k=10$ is constant. Now we will bound $S_1$ and $S_2$. We start with $S_1$ using \cref{Cor:SzemTrot}:
			\begin{align*}
				S_1 &= \sum_{l \geq k} \sum_{\substack{p \in F_{0,B}(\cH) \\ 2^l \leq a(p) < 2^{l+1} \\ a(p) < \sqrt{n}}} \begin{pmatrix} a(p) \\ 2 \end{pmatrix} 
				\leq \sum_{\substack{l \geq k \\ 2^{l+1}<\sqrt{n}}}  t^{\ast}(n,2^l) \begin{pmatrix} 2^{l+1} \\ 2 \end{pmatrix} \\
				&= \sum_{\substack{l \geq k \\ 2^{l+1} < \sqrt{n}}}  t^{\ast}(n,2^l) 2^l (2^{l+1}-1)		
				\leq \sum_{\substack{l \geq k \\ 2^{l+1} < \sqrt{n}}} \beta \frac{n^2}{2^{3l}} 2^l (2^{l+1}-1)\\
				&\leq 2 \beta n^2 \sum_{l \geq k} \frac{1}{2^l} = \frac{4 \beta}{2^k} n^2 \leq \frac{1}{8} \begin{pmatrix} n \\ 2\end{pmatrix},
			\end{align*}
			Since $\beta = 2.5^3$, $k=10$ and $n\geq 2$. 
			For the next sum we use \cref{Lem:Beck2.1} and \cref{Cor:SzemTrot}:
			\begin{align*}
				S_2 &= \sum_{l \geq0} \sum_{\substack{p \in F_{0,B}(\cH) \\ 2^l \sqrt{n} \leq a(p) < 2^{l+1} \sqrt{n} \\ a(p) < \frac{n}{100}}} \begin{pmatrix} a(p) \\ 2 \end{pmatrix} \leq \sum_{\substack{l \geq 0\\ 2^{l} \sqrt{n} < \frac{n}{100}}} t(n, 2^l \sqrt{n}) \begin{pmatrix} 2^{l+1} \sqrt{n} \\ 2 \end{pmatrix}\\
				&= t(n, \sqrt{n}) \begin{pmatrix} 2 \sqrt{n} \\ 2 \end{pmatrix} +\sum_{\substack{l \geq 1\\ 2^{l} \sqrt{n} < \frac{n}{100}}} t(n,\underbrace{2^l \sqrt{n}}_{> \sqrt{2n}}) \begin{pmatrix} 2^{l+1} \sqrt{n} \\ 2 \end{pmatrix}\\
				&< \beta \frac{n^2}{n^{3/2}} \sqrt{n} (2 \sqrt{n}-1) + \sum_{\substack{l \geq 1\\ 2^{l} \sqrt{n} < \frac{n}{100}}} \frac{2n}{2^l \sqrt{n}} 2^l \sqrt{n} (2^{l+1} \sqrt{n}-1)\\
				&< 2 \beta n^{3/2} + 4 n^{3/2} \sum_{\substack{l \geq 1 \\ 2^{l} < \frac{\sqrt{n}}{100}}} 2^l = 2 \beta n^{3/2} + 4 n^{3/2} \left( \frac{\sqrt{n}}{50}-2\right)\\
				&= \frac{2}{25} n^2 + (2 \beta -8 )n^{3/2} \leq \frac{1}{4} \begin{pmatrix} n \\ 2\end{pmatrix}.
			\end{align*}
			Combining the two results we get a lower bound for $S_3$
			\begin{equation*}
				S_3 \geq \vert F_1 \vert \cdot \vert F_2 \vert - \frac{1}{4}\begin{pmatrix} n \\ 2 \end{pmatrix} - \frac{1}{8}\begin{pmatrix} n \\ 2 \end{pmatrix} \geq \frac{1}{16}n^2.
			\end{equation*}
			So now assume that condition (a) of the theorem does not hold, then
			\begin{equation*}
				\vert F_{0,B}(\cH) \vert \geq \sum_{\substack{p \in F_{0,B}\\ 2 \leq a(p) < 2^k}} 1 \geq \begin{pmatrix} 2^k \\ 2 \end{pmatrix}^{-1} \sum_{\substack{p \in F_{0,B}\\ 2 \leq a(p) < 2^k}} \begin{pmatrix} a(p) \\ 2 \end{pmatrix} \geq \begin{pmatrix} 2^k \\ 2 \end{pmatrix}^{-1} \frac{1}{16} n^2.
			\end{equation*}
			Since we have seen that $k=10$ is a possible choice the constant would be $c=\frac{1}{8 380416}$.
		\end{proof}
	\end{Theorem}
	
	\begin{Remark}
		In condition (a) the constant $\frac{1}{100}$ is by no means optimal, but since for us the constant plays an insignificant role we stick to the original constant used by Beck.\par 
		Further notice that we have even proved a stronger theorem, namely that
		\begin{equation*} 
			\big\vert \{p \in F_{0,B}(\cH) \mid 2 \leq a(p) \leq 2^k\}\big\vert \geq c \cdot n^2,
		\end{equation*}
		if condition (a) does not hold.
	\end{Remark}

	\begin{proof}[Proof of \cref{Thm:BeckInduction}]
		The idea of the proof is that for a family $F_i$ of hyperplanes, the other families induce a hyperplane arrangement in all $H \in F_i$, thus we can conclude by induction. The case $d=2$ is already completed by \cref{Thm:DualBeck}, where we even proved the stronger statement that 
		\begin{equation*} 
			\vert \{p \in F_{0,B}(\cH) \mid 2 \leq a(p) < 2^k\}\vert \geq c_2 \cdot n^2
		\end{equation*}
		if for all $p \in B$ we have $a(p) < \frac{n}{100}$. That $a(p) < \frac{1}{100} n$ is guaranteed by the assumption that at most $c \cdot \vert F_i \vert$ hyperplanes from $F_i$ can intersect in one point and $c < \frac{1}{100}$. So the initial case of the induction holds.\par\medskip
		
		Now consider the family $F_1$, we are interested in the $d-1$ dimensional arrangement which is induced on the hyperplanes $H\in F_1$. Notice that for $H_i \in F_i, H_j \in F_j$ and $H_k\in F_k$, $i \neq j \neq k$, we have $H_i \cap H_j \neq H_i \cap H_k$ since otherwise we get a contradiction to the assumption that $B \cap \bigcap_{i=1}^d f_i = \{p\}$ for all $(f_1,...,f_d) \in F_1 \times ... \times F_d$. So the different families induce different $(d-2)$-hyperplanes on the hyperplanes of $F_1$. Set
		\begin{equation*}
			\cH^{H \ast} := \{H \cap f \mid f \in F_2 \cup ... \cup F_d\},
		\end{equation*}
		here $H \cap f \neq \emptyset$ holds for all $f$ by the assumption on the intersection behaviour.  Now we prove the following claim, which gives us the induction hypothesis:\par\medskip 
		
		\textbf{Claim:} $\vert \cH^{H \ast} \vert > \delta \cdot n$, for some constant $\delta$, and at least $\epsilon \cdot \frac{n}{d}$ hyperplanes $H \in F_1$, for some constant $\epsilon >0$.\par \medskip 
		We proof the claim: We do this by considering two families and show that the one induces enough planes on the second one. To do so let $P$ be a generic $2$-dimensional plane in $\RR^d$, i.e. $P \cap H$ is one dimensional for all $H \in F_1 \cup F_2$ and they are all distinct for different $H$. And $P\cap f$ is a point for all $f=H \cap K$ with $H \in F_1$, $K \in F_2$ which are also distinct if the $K \cap H$ are distinct. So each hyperplane corresponds to a line and each $d-2$ dimensional flat corresponds to a point. The intersection behaviour for the lines clearly fulfils the assumptions in \cref{Thm:DualBeck}. So we can apply the theorem and get $c \cdot \left(\frac{n}{d}\right)^2$ intersections points and therefore $c \cdot \left(\frac{n}{d}\right)^2$ induced flats. Since each hyperplane can carry at most $\frac{n}{d}$ induced flats we see that the flats have to spread out such that the claim holds.\par\medskip
		
		Denote the set of hyperplanes from $F_1$ for which the claim holds by $\tilde{F_1}$. It is clear that the intersection behaviour of the different families also holds in the $d-1$ dimensional arrangement induced on the hyperplanes in $\tilde{F_1}$. Also assume that the stronger statement 
		\begin{equation*} 
			\vert \{p \in F_{0,B}(\cH) \mid l \leq a(p) < l^k\}\vert \geq c_l \cdot n^l
		\end{equation*}
		is proved for all dimensions $l$ up to $d-1$, where $k$ is a constant.\par\medskip 
		
		Now we can do the induction step. We get the following inequality
		\begin{align*}
			\vert F_{0,B}(\cH) \vert \cdot \begin{pmatrix} d^k \\ d \end{pmatrix} > \sum_{\substack{p \in F_{0,B}(\cH)\\ d \leq a(p) < d^{k}}} \begin{pmatrix} a(p) \\ d \end{pmatrix} \geq \sum_{H \in F_1} \sum_{\substack{p \in F_{0,B}(\cH^{H\ast})\\ d-1 \leq a_{\cH^{H\ast}}(p) < (d-1)^k}} \begin{pmatrix} a_{\cH^{H\ast}}(p) \\ d-1 \end{pmatrix}.
		\end{align*}
		For the last inequality notice that $a_{\cH^{H\ast}}(p)$ now only counts the hyperplanes in $\cH^{H \ast}$ and we only have to take $d-1$ out of them since we fixed the choice $H \in F_1$. Now further by the induction assumption 
		\begin{align*}
			\sum_{H \in F_1} &\sum_{\substack{p \in F_{0,B}(\cH^{H\ast})\\ d-1 \leq a(p) < (d-1)^k}} \begin{pmatrix} a(p) \\ d-1 \end{pmatrix} 
			\geq \sum_{H \in \tilde{F_1}} \sum_{\substack{p \in F_{0,B}(\cH^{H\ast})\\ d-1 \leq a(p) < (d-1)^k}} \begin{pmatrix} a(p) \\ d-1 \end{pmatrix}\\
			&\geq \sum_{H \in \tilde{F_1}} \vert \{ p \in F_{0,B}(\cH^{H\ast}) \mid d-1 \leq a(p) < (d-1)^k\}\vert\\
			&\geq \sum_{H \in \tilde{F_1}} c_{d-1} \cdot \delta^{d-1} n^{d-1} \geq \epsilon \frac{n}{d} c_{d-1} \delta^{d-1} n^{d-1} = \frac{\epsilon c_{d-1}}{d} \cdot \delta^{d-1} n^d.\\
		\end{align*}
		This finally yields
		\begin{equation*}
			\big\vert \{ F_{0,B}(\cH) \mid d \leq a(p) < d^k \} \big\vert \geq \begin{pmatrix} d^k \\ d \end{pmatrix}^{-1} \frac{\epsilon c_{d-1}}{d} \cdot \delta^{d-1} n^d =: c_d n^d.
		\end{equation*}
	\end{proof}
	
	\begin{appendices}
		\section{FLC in non-abelian lcsc groups}\label[appendix]{Appendix:FLC}
		This appendix is dedicated to give some more information on sets with finite local complexity for lcsc groups and we will show that all model sets have FLC.\par\medskip
		
		A locally finite subset $\Lambda \subset G$ which fulfils one and therefore all of the conditions in the following lemma has \textit{finite local complexity} as defined in \cref{Rem:FLC}. In the lemma we only need that $\Lambda$ is locally finite, so we could also define the term of finite local complexity for this type of sets.
		
		\begin{Lemma}[Finite local complexity] \label{FLC_Lemma}
			Let $G$ be a lcsc group and $\Lambda \subset G$ a locally finite set, i.e. for all bounded $B\subset G$ we have that $B \cap \Lambda$ is finite. Then the following are equivalent:
			\begin{enumerate}[label=\roman*)]
				\item For all $B \subset G$ bounded there exists a finite $F_B \subset G$ such that
				\begin{equation*}
					\forall g \in G \,\exists h \in \Lambda^{-1}\Lambda \,\exists f \in F_B: (B g^{-1}\cap \Lambda)h=Bf^{-1}\cap \Lambda.
				\end{equation*} 
				\item For all $B \subset G$ bounded there exists a finite $F_B \subset G$ such that
				\begin{equation*}
					\forall g \in G \,\exists h \in G \,\exists f \in F_B: (B g^{-1}\cap \Lambda)h=Bf^{-1}\cap \Lambda.
				\end{equation*}
				\item $\Lambda \Lambda^{-1}$ is locally finite.
				\item For all $B \subset G$ bounded: 
				\begin{equation*}
					\big\vert \{B \cap \Lambda \lambda^{-1}\mid \lambda \in \Lambda\}\big\vert< \infty.
				\end{equation*}
				\item The complexity function $p(r)$ is finite for all $r\geq 0$.
			\end{enumerate}
			\begin{proof}
				First we will show the equivalence of $i), ii)$ and $iii)$. Afterwards we will show the equivalence of $iii)$ and $iv)$. Finally the will show the equivalence of $iv)$ and $v)$.\par
				
				$i) \Rightarrow ii)$: This step is obvious, since $\Lambda^{-1}\Lambda \subset G$.\par
				
				$ii) \Rightarrow iii)$: Without loss of generality we can assume that $B$ is compact and contains the identity, otherwise we just simply expand $B$ and notice that this would just increase the number of elements in the intersection. For this $B$ we choose $F_B$ such that $ii)$ holds. Since $F_B$ is finite and $B$ is bounded we see that $B':=BF_B^{-1}$ is also bounded, further we see, since $\Lambda$ is locally finite, that $F:=B'\cap \Lambda$ is finite.\\
				Now let $\lambda_1, \lambda_2 \in \Lambda$ be arbitrary with $\lambda_1\lambda_2^{-1} \in B$. We get $\lambda_1 \in B\lambda_2 \cap \Lambda$ and since we assumed $e \in B$ we also get $\lambda_2 \in B \lambda_2 \cap \Lambda$. With our assumption we get that $h_1 \in G$ and $f_1 \in F_B$ exist with
				\begin{equation*}
					(B \lambda_2 \cap \Lambda)h_1=B f_1^{-1} \cap \Lambda.
				\end{equation*} 
				Putting the pieces together we obtain
				\begin{equation*}
					\{\lambda_1,\lambda_2\} h_1 \subseteq (B \lambda_2 \cap \Lambda) h_1 = B f_1^{-1}\cap \Lambda \subset B F_B^{-1} \cap \Lambda = B'\cap \Lambda =F.
				\end{equation*}
				So $\lambda_1 \lambda_2^{-1} = (\lambda_1 h_1^{-1})(\lambda_2 h_1^{-1})^{-1}\in F F^{-1}$ and we get that $\Lambda\Lambda^{-1} \cap B \subset F F^{-1}$ is finite.\par\medskip
				
				$iii) \Rightarrow i)$: Let $B \subset G$ be bounded. Without loss of generality we can assume $B$ to be symmetric, i.e. $B=B^{-1}$. Since $B$ is bounded $B^2$ is also bounded and $B^2 \cap \Lambda \Lambda^{-1}$ is finite by assumption. Then
				\begin{equation*}
					BB \cap \Lambda \Lambda^{-1} = \bigcup_{b \in B} \bigcup_{\lambda \in \Lambda} Bb \cap \Lambda \lambda^{-1},
				\end{equation*}
				and since we know that this is finite we conclude that $Bb \cap \Lambda \lambda^{-1}$ can only have finitely many different forms. So we find $b_1,...,b_s \in B$ and $\lambda_1,...,\lambda_t \in \Lambda$ such that for arbitrary $b\in B$ and $\lambda \in \Lambda$ there exists a $n\in \{1,...,s\}$ and a $m\in \{1,...,t\}$ with
				\begin{equation*}
					Bb \cap \Lambda \lambda^{-1} = Bb_n \cap \Lambda \lambda_m^{-1}.
				\end{equation*}
				Let $g\in G$ be arbitrary. Then the two following cases can appear:\par \medskip
				\textbf{Case 1}: $B g^{-1} \cap \Lambda = \emptyset$, to deal with this case we simply set $f_0 = g$, for one such $g$.\par \medskip
				\textbf{Case 2:} $B g^{-1} \cap \Lambda \neq \emptyset$, then there exists a $b'\in B$ such that $b'g^{-1} = \lambda$. Set $b:=b'^{-1}$, then $g^{-1}=b\lambda$ and, since $B$ is symmetric, $b \in B$. Now choose $n$ and $m$ such that $Bb \cap \Lambda \lambda^{-1}=B b_n \cap \Lambda \lambda_n^{-1}$ and set $h:=\lambda^{-1}\lambda_n \in \Lambda^{-1} \Lambda$. Now we get
				\begin{align*}
					(Bg^{-1} \cap \Lambda)h &= (Bb \lambda \cap \Lambda) \lambda^{-1} \lambda_m = (Bb \cap \Lambda \lambda^{-1}) \lambda_m\\
					&=(B b_n \cap \Lambda \lambda_m^{-1}) \lambda_m = B b_n \lambda_m \cap \Lambda
				\end{align*}
				Finally we set $F_B':=\big\{\lambda_m^{-1} b_n^{-1} \mid n \in \{1,...,s\}, m\in\{1,...,t\}\big\}$, which is finite.\par\medskip
				
				To combine both cases define $F_B := F_B' \cup {f_0}$.\par \medskip
				
				$iv) \Rightarrow iii)$: Let $B\subset G$ be bounded. For $\{B \cap \Lambda \lambda^{-1} \mid \lambda \in \Lambda\}$ we use our assumption to find finitely many representatives $\lambda_1,...,\lambda_k$ such that 
				\begin{equation*}
					\{B \cap \Lambda \lambda^{-1} \mid \lambda \in \Lambda\} = \bigcup_{l=1}^k \{B \cap \Lambda \lambda_l^{-1}\}.
				\end{equation*}
				We get 
				\begin{equation*} 	
					B \cap \Lambda \Lambda^{-1} = \bigcup_{\lambda \in \Lambda} B \cap \Lambda \lambda^{-1} =\bigcup_{l=1}^k B \cap \Lambda \lambda_l^{-1}. 
				\end{equation*}
				The sets $B \cap \Lambda \lambda_l^{-1} = (B \lambda_l \cap \Lambda)\lambda_l^{-1}$ are finite, since $\Lambda$ is locally finite.\medskip
				
				$iii) \Rightarrow iv)$: Let $B \subset G$ be bounded, then by our assumption $B \cap \Lambda\Lambda^{-1}$ is finite. Further
				\begin{equation*}
					B \cap \Lambda\Lambda^{-1} = \bigcup_{\lambda \in \Lambda}  B \cap \Lambda \lambda^{-1}.
				\end{equation*}
				Since the left hand side is finite this also holds for the right hand side. But this means that there can only be finitely many combinations for the sets $B \cap \Lambda \lambda^{-1}$. So we get
				\begin{equation*}
					\vert \{ B \cap \Lambda \lambda^{-1} \mid \lambda \in\Lambda\} \vert < \infty.
				\end{equation*}\medskip
				
				$iv) \Leftrightarrow v)$: This is obvious since each ball $B_r(e)$ is a bounded set and on the other hand for every bounded set $B$ we can find a $r>0$ such that $B \subset B_r(e)$ holds.  
			\end{proof}
		\end{Lemma}
		
		The following proposition justifies our restriction to precompact windows with non-empty interior.
		
		\begin{Proposition}\label{Prop:CPSFLCrelativelydense}
			Let $(G,H, \Gamma)$ be a CPS, $W \subset H$ a subset and $\Lambda:= \pi_G((G \times W)\cap \Gamma)$.
			\begin{enumerate}[label=\roman*)]
				\item If $W^\circ \neq \emptyset$, then $\Lambda$ is relatively dense,
				\item if $W$ is relatively compact, then $\Lambda$ is uniformly discrete,
				\item if $W$ is relatively compact and $W^\circ \neq \emptyset$, then $\Lambda$ has FLC.
			\end{enumerate}
			\begin{proof}
				\begin{enumerate}[label=\roman*)]
					\item We are using \cref{Lem:ProduktOffenKompakt}. Since ${W^\circ \neq \emptyset}$ this also holds for the inverse ${(W^{-1})^\circ \neq \emptyset}$ and we can choose an open subset $\emptyset \neq U \subset W^{-1}$. By the \cref{Lem:ProduktOffenKompakt} we find a compact set $K$ such that $G\times H= (K \times U)\Gamma$. Let $g \in G$ be arbitrary, we can find $u \in U, k \in K$ and $\gamma \in \Gamma$ such that
					\begin{equation*}
						(g, e_H) = (k,u)(\gamma_{G}, \gamma_{H}).
					\end{equation*}
					This tells us that $u \gamma_{H} = e_H$ and therefore $\gamma_{H}=u^{-1} \in (W^{-1})^{-1}=W$, so $\gamma_{G} \in \Lambda$. Therefore $g = k \gamma_{G} \in K \Lambda$. This shows the claim.
					
					\item Let us assume $\Lambda$ is not uniformly discrete, then for all $r>0$ there exists ${x,y \in \Lambda}$ such that ${d(x,y)<r}$. By the right-invariance of $d$ this is equivalent to ${d(e, yx^{-1})<r}$. We can lift $x$ and $y$ to elements in the product and get
					\begin{equation*}
						\pi_G\vert_\Gamma^{-1}(x)=:(x_G,x_H),\, \pi_G\vert_\Gamma^{-1}(y)=:(y_G,y_H) \in \Gamma \cap (G \times W). 
					\end{equation*}
					Since $\Gamma$ and $G$ are groups we can deduce 
					\begin{equation*}
						{(y_G,y_H)(x_G^{-1},x_H^{-1}) \in \Gamma \cap (G \times WW^{-1})}.
					\end{equation*}
					Since we know that ${yx^{-1} \in B_r(x)}$ we get 
					\begin{equation*}
						{(y_G,y_H)(x_G^{-1},x_H^{-1}) \in \Gamma \cap (B_r(e_G) \times WW^{-1})}.
					\end{equation*}
					Since $W$ is relatively compact, $WW^{-1}$ is also relatively compact and therefore bounded. Moreover $B_r(e_G)$ is bounded so the product ${B_r(e_G)\times WW^{-1}}$ is bounded. Thus, since $\Gamma$ is a lattice, we get that ${\Gamma \cap (B_r(e_G) \times WW^{-1})}$ is finite. By the injectivity of ${\pi_G}$ we know that ${d(a_G,b_G) \neq 0}$ for ${a \neq b \in \Gamma}$ so we get that $d(a_G,b_G) > 0$ for $a,b \in \Gamma \cap (B_r(e_G) \times WW^{-1})$ and by finiteness there is a minimal distance $\tilde{d}$. Now set $\tilde{r}< \tilde{d}$ and conclude ${\Gamma \cap (B_{\tilde{r}}(e_G) \times WW^{-1})=\{(e_G,e_H)\}}$. This is a contradiction to the assumption, since we do not find two elements, which are this close together. Therefore $\Lambda$ has to be uniformly discrete for $\tilde{r}$.
					
					\item By i) and ii) we know that $\Lambda$ is a Delone set and therefore locally finite. We want to use the characterisation iii) of \cref{FLC_Lemma}, so we show that $B \cap \Lambda \Lambda^{-1}$ is finite for a bounded set $B \subset H$. It is enough to show that the preimage of this set is finite. Since taking the preimage and intersecting commutes we get
					\begin{equation*}
						\pi_G^{-1}(B \cap \Lambda\Lambda^{-1}) = \pi_G^{-1}(B) \cap \pi_G^{-1}(\Lambda \Lambda^{-1}).
					\end{equation*}
					Now we can consider the two parts separately and then intersect them, so the preimage of $B$ is obviously $\pi_G^{-1}(B)=B \times H$.\\
					For the second part we need to remember the definition of $\Lambda$, this was given by ${\Lambda=\pi_G((G \times W)\cap \Gamma)}$, so
					\begin{equation*}
						\pi_G^{-1}(\Lambda \Lambda^{-1}) = \pi_G^{-1}(\pi_G((G \times W)\cap \Gamma) \pi_G((G \times W)\cap \Gamma)^{-1}).
					\end{equation*} 
					We want to show that this is a subset of $\Gamma \cap (G \times WW^{-1})$. So let $\lambda_1, \lambda_2 \in \Lambda$ then they are both in $\Gamma_G$ and therefore $\lambda_1 \lambda_2^{-1} \in \Gamma_G$ and there exists a unique preimage inside $\Gamma$ which we name $(\lambda_1 \lambda_2^{-1},x)$. On the other hand the preimage of $\lambda_i$, $i\in\{1,2\}$, is $(\lambda_i, w_i) \in \Gamma \cap (G \times W)$. And
					\footnotesize\begin{equation*}
						(\lambda_1, w_1)(\lambda_2, w_2)^{-1} = (\lambda_1, w_1)(\lambda_2^{-1}, w_2^{-1}) = (\lambda_1 \lambda_2^{-1}, w_1w_2^{-1}) \in \Gamma \cap (G \times WW^{-1}).
					\end{equation*}\normalsize
					Since the preimage was unique and we see that 
					\begin{equation*}
						\pi_G(\lambda_1 \lambda_2^{-1}, w_1w_2^{-1})= \lambda_1 \lambda_2^{-1}
					\end{equation*}
					we get that $\pi_G^{-1}(\lambda_1 \lambda_2^{-1}) \in \Gamma \cap (G \times WW^{-1})$.\\
					Combining the two arguments we get
					\small\begin{equation*}
						\pi_G^{-1}(B \cap \Lambda\Lambda^{-1}) = (B \times H) \cap \Gamma \cap (G \times WW^{-1}) = \Gamma \cap (B \times WW^{-1}).
					\end{equation*}\normalsize
					
					Since $W$ is relatively compact we get that $\overline{W}$ is compact. Since $W \subset \overline{W}$ we see $W \subset H$ is bounded. Hence there exists a $r>0$ such that $r>d(w_1, w_2)$ for all ${w_1,w_2 \in W}$. And once more by right-invariance of the metric we get ${r>d(w_1w_2^{-1},e)}$. This tells us that $WW^{-1} \subset B_r(e)$ and therefore it is bounded. Further $B \subset G$ was a bounded set. We see that $B \times WW^{-1} \subset G \times H$ is bounded in the product. Since $\Gamma$ is a lattice it has FLC and therefore $(B \times WW^{-1}) \cap \Gamma$ is finite.
				\end{enumerate}
			\end{proof}
		\end{Proposition}
		
		\section{Homogenous Lie groups}\label[appendix]{Appendix:Poly}
		In the first part of this appendix we are concerned with the proof of \cref{Thm:NonCrooked}.
		\begin{Proposition}
			Every locally two-step nilpotent homogeneous Lie group is non-crooked.
			\begin{proof}
				This follows directly by considering the BCH-formula and noticing that by using the assumption on locally two-step nilpotent Lie groups that for all $X,Y \in \fg$
				\begin{equation*}
					X \ast Y = X + Y + \frac{1}{2}[X,Y].
				\end{equation*}
				So for a hyperplane $H =v_0 + \sum_{i=1}^d t_i v_i$, with $t_i \in \RR$, $v_i \in \RR^d$ and $d$ the dimension of the Lie group. We get for all $X \in \fg$
				\small\begin{align*}
					X \ast H &= X + H + \frac{1}{2}[X,H] = v_0 +X +\sum_{i=1}^d t_i v_i + \frac{1}{2} [X, v_0] + \frac{1}{2} \sum_{i=1}^d t_i [X,v_i]\\
					&= v_0 +X + \frac{1}{2} [X, v_0] +\sum_{i=1}^d t_i (v_i + [X,v_i]) =: \tilde{v_0} + \sum_{i=1}^d t_i \tilde{v_i}.
				\end{align*}\normalsize
				And this is again a hyperplane.
			\end{proof}
		\end{Proposition}
		
		At first sight the condition of locally two-step nilpotent seems weaker than the condition of being two-step nilpotent, but in fact the two are equivalent by the following proposition.
		
		\begin{Proposition}
			Let $G$ be a Lie group. If $G$ is locally two-step nilpotent then $G$ is two-step nilpotent or abelian. If $G$ is two-step nilpotent it is locally two-step nilpotent.
			\begin{proof}
				The conclusion from two-step nilpotent to locally two-step nilpotent is trivial, so two-step nilpotent Lie groups are locally two-step nilpotent Lie groups.\par 
				So now assume that we have a locally two-step nilpotent Lie group and arbitrary $X,Y,Z \in \fg$, then
				\begin{align*}
					0 &= [X+Y,[X+Y,Z] = [X,[X,Z]]+[X,[Y,Z]]+[Y,[X,Z]]+[Y,[Y,Z]]\\
					&= [X,[Y,Z]]+[Y,[X,Z]].
				\end{align*}
				This means that $[X,[Y,Z]] =-[Y,[X,Z]] = [Y,[Z,X]]$. By using Jacobi's identity we have
				\begin{equation*}
					0 = [X,[Y,Z]]+[Y,[Z,X]]+[Z,[X,Y]].
				\end{equation*}
				And therefore by using the equality we found before we have $2[Y,[Z,X]]=[Z,[Y,X]]$. Since $X$, $Y$ and $Z$ are arbitrary, we can switch the roles of $Y$ and $Z$. Thus \linebreak ${2[Z,[Y,X]]=[Y,[Z,X]]}$. So in total this means
				\begin{equation*}
					[Z,[Y,X]]=2[Y,[Z,X]] = 4 [Z,[Y,X]],
				\end{equation*}
				and therefore $[Z,[Y,X]]=0$. So $G$ is two-step nilpotent or abelian.
			\end{proof}	
		\end{Proposition}
		We have seen that the class of non-crooked homogeneous Lie groups contains the abelian and the two-step nilpotent homogeneous Lie groups. We will now see that a higher nilpotency degree always implies crookedness.
		
		\begin{Proposition}
			A three-step homogeneous Lie group is crooked.
			\begin{proof}
				Let $H$ be a hyperplane given by $H=v_0 + \sum_{i=1}^n t_i v_i$, with $t_i \in \RR$, $v_i \in \RR^d$ and $d$ the dimension of the Lie group. We get for all $X \in \fg$
				\footnotesize\begin{align*}
					X \ast H &= X + H + \frac{1}{2}[X,H] + \frac{1}{12}([X,[X,H]]-[H,[X,H]])\\
					&= X+ v_0 + \frac{1}{2}[X,v_0]+\frac{1}{12}\left([X,X,v_0]+[v_0,[X,v_0]]\right)\\
					&\quad\quad + \sum_{i=1}^n t_i\cdot \left( v_i+\frac{1}{2} [X,v_i]+ \frac{1}{12}([X,[X,v_i]] - [v_0,[X,v_i]] - [v_i,[X,v_0]]) \right)\\
					&\quad\quad - \frac{1}{12} \sum_{i=1}^n \sum_{j=1}^n t_i t_j [v_i,[X,v_j]].
				\end{align*}\normalsize
				So for this to be non-crooked the last sum has to disappear, i.e.
				\begin{equation*}
					\sum_{i=1}^n \sum_{j=1}^n t_i t_j [v_i,[X,v_j]] = 0.
				\end{equation*}
				But since the $t_i$, $t_j$ are parameters we can only compare the summands with the same coefficients, so for all $i,j \in \{1,...,n\}$, $i\neq j$
				\begin{equation*}
					[v_i,[X,v_j]] + [v_j, [X,v_i]] =0
				\end{equation*}
				and for the diagonal, i.e. $i=j$
				\begin{equation*}
					[v_i,[X,v_i]]=0.
				\end{equation*}
				But this last conditions is the locally two-step nilpotency condition, which as we have seen implies two-step nilpotency.
			\end{proof} 
		\end{Proposition}
		
		\begin{Corollary}
			All nilpotent homogeneous Lie groups, with nilpotency degree greater than two, are crooked. 
		\end{Corollary}
		
		\subsection{Form of the polynomial action}\label[appendix]{Appendix:Poly2}~\par\medskip
		
		We can find some restrictions on the form of the polynomials in the group law. Since we have a dilation structure we get
		\begin{align*}
			(r^{\nu_1}P_1(x,y),&...,r^{\nu_n}P_n(x,y))=D_r( P_1(x,y),..., P_n(x,y)) = D_r(xy)\\
			&= D_r(x) D_r(y) = ( P_1(D_r(x),D_r(y)),..., P_n(D_r(x),D_r(y))).
		\end{align*}
		This means for the polynomials we have
		\begin{equation*}
			P_i(D_r(x),D_r(y))= r^{\nu_i}P_i(x,y).
		\end{equation*}
		This gives us a restriction on the form of the $P_i$ namely if
		\begin{equation*}
			P_i(x,y)= \sum_{\alpha_1,...,\alpha_n,\beta_1,...,\beta_n \in \NN} c_{(\alpha_1,...,\alpha_n,\beta_1,...,\beta_n)} x_1^{\alpha_1}...x_n^{\alpha_n} y_1^{\beta_1}...y_n^{\beta_n} 
		\end{equation*}
		then for all summands with $c\neq 0$ it is $\nu_i=\nu_1\alpha_1+...+\nu_n\alpha_n+\nu_1\beta_1+...+\nu_n\beta_n$. Since we sorted the $\nu_i$ by their size we see that in the $i$-th entry all $\alpha_j$ and $\beta_j$ have to be zero if $\nu_j > \nu_i$. This means that for the Polynomial $P_i$ we have that it is only dependent on the first $i$ entries of both $x$ and $y$.\par
		Another restriction which can be seen from the BCH formula is that the polynomials are of a certain form, namely
		\begin{equation*}
			P_i(x,y) = x_i+y_i + \sum_{\substack{\alpha_1,...,\alpha_n,\beta_1,...,\beta_n \in \NN\\\sum_i \alpha_i \neq 0, \sum_i \beta_i \neq 0 }} c_{(\alpha_1,...,\alpha_n,\beta_1,...,\beta_n)} x_1^{\alpha_1}...x_{i-1}^{\alpha_{i-1}} y_1^{\beta_1}...y_{i-1}^{\beta_{i-1}}. 
		\end{equation*}
		This means that in the polynomial all summands except the two in front have entries from $x$ and $y$. \par
		
		Notice that in the non-crooked case there can be no product of the $t_i$, this means that we have either one of the $\beta_j$ is one and all the overs are zero or all $\beta_j$ are zero.\par \medskip

		To finish this appendix we give one example of such a group. Consider the Heisenberg group $\HH$ of upper triangular matrices with ones on the diagonal. The entries of the second diagonal have weight 1 and the third diagonal has weight 2. For two elements 
		\begin{equation*}
			\begin{pmatrix} 1 & a & c \\ 0 & 1 & b \\ 0 & 0 & 1	\end{pmatrix}, \begin{pmatrix} 1 & x & z \\ 0 & 1 & y \\ 0 & 0 & 1	\end{pmatrix}
		\end{equation*}
		the polynomials then are given by
		\begin{align*}
			&P_1((a,b,c),(x,y,z))= a+x,\\
			&P_2((a,b,c),(x,y,z))= b+y,\\
			&P_3((a,b,c),(x,y,z))= c+z + ay. 
		\end{align*}
		And we see that $\HH$ is $2$-step nilpotent and therefore non-crooked.\par \medskip 
			
	\end{appendices}
	
\printbibliography
\end{document}